\documentclass[12pt, reqno]{amsart}
\usepackage{amssymb, amsthm, amsmath, amsfonts}
\usepackage{array, epsfig}
\usepackage{bbm}
\usepackage{color}

  \evensidemargin
\oddsidemargin
\newcommand{\Sym}{\text{\rm Sym}}

\DeclareMathOperator*{\Ker}{Ker}

\newcommand{\E}{\mathbb E}

\newcommand{\R}{\mathbb{R}}
\newcommand{\N}{\mathbb{N}}

\renewcommand{\P}{\mathbb{P}}

\newcommand{\conv}{\mathop{\mathrm{Conv}}\nolimits}

\newcommand{\eps}{\varepsilon}
\newcommand{\eqdistr}{\stackrel{d}{=}}
\newcommand{\eqdef}{\stackrel{def}{=}}

\newcommand{\ind}{\mathbbm{1}}

\newcommand{\dd}{{\rm d}}

\theoremstyle{plain}
\newtheorem{theorem}{Theorem}[section]
\newtheorem{lemma}[theorem]{Lemma}
\newtheorem{corollary}[theorem]{Corollary}

\theoremstyle{definition}

\newtheorem{example}[theorem]{Example}

\theoremstyle{remark}
\newtheorem{remark}[theorem]{Remark}

\newcommand{\stirling}[2]{\genfrac{[}{]}{0pt}{}{#1}{#2}}

\begin{document}

\author{Zakhar Kabluchko}
\address{Zakhar Kabluchko, Institut f\"ur Mathematische Stochastik,
Universit\"at M\"unster,
Orl\'eans--Ring 10,
48149 M\"unster, Germany}
\email{zakhar.kabluchko@uni-muenster.de}

\author{Vladislav Vysotsky}
\address{Vladislav Vysotsky, University of Sussex, Pevensey 2 Building, BN1 9RH Brighton,
United Kingdom and St.\ Petersburg Department of Steklov Mathematical Institute, Fontanka~27,
191011 St.\ Petersburg,
Russia}
\email{v.vysotskiy@sussex.ac.uk, vysotsky@pdmi.ras.ru}

\author{Dmitry Zaporozhets}
\address{Dmitry Zaporozhets, St.\ Petersburg Department of Steklov Mathematical Institute,
Fontanka~27,
191011 St.\ Petersburg,
Russia}
\email{zap1979@gmail.com}

\title[A multidimensional analogue of the arcsine law]{A multidimensional analogue of the arcsine law for the number of positive terms in a random walk}

\begin{abstract}
Consider a random walk $S_i= \xi_1+\ldots+\xi_i$, $i\in\N$, whose increments $\xi_1,\xi_2,\ldots$ are independent identically distributed random vectors in $\R^d$ such that $\xi_1$ has the same law as $-\xi_1$ and $\P[\xi_1\in H] = 0$
for every affine hyperplane $H\subset \R^d$.
%For the random variable
%$$
%M_{n,k}^{(d)} := \frac{1}{2} \sum_{1\leq i_1 < \ldots < i_k\leq n} \ind_{\{0\notin \conv(S_{i_1},\ldots, S_{i_k})\}},
%$$
%where $1\leq k\leq n$,
%we prove the distribution-free formula
%$$
%\E M_{n,k}^{(d)} = \binom n k \frac {B(k, d-1) + B(k, d-3) +\ldots} {2^k k!},
%$$
Our main result is the distribution-free formula 
$$
\E \left[\sum_{1\leq i_1 < \ldots < i_k\leq n} \mathbbm{1}_{\{0\notin \conv(S_{i_1},\ldots, S_{i_k})\}}\right] = 2 \binom n k \frac {B(k, d-1) + B(k, d-3) +\ldots} {2^k k!},
$$
where the $B(k,j)$'s are defined by their generating function
$
(t+1) (t+3) \ldots (t+2k-1) = \sum_{j=0}^{k} B(k,j) t^j.
$
The expected number of $k$-tuples above admits the following geometric interpretation: it is the expected number of $k$-dimensional faces of a randomly and uniformly sampled open Weyl chamber of type $B_n$ that are not intersected by a generic linear subspace $L\subset \R^n$ of codimension $d$.
The case $d=1$ turns out to be equivalent to the classical discrete arcsine law for the number of positive terms in a one-dimensional random walk with continuous symmetric distribution of increments.
We also prove similar results for random bridges with no central symmetry assumption required.
\end{abstract}


\thanks{This paper was written when V.V. was affiliated to Imperial College London, where his work was supported by People Programme (Marie Curie Actions) of the European Union's Seventh Framework Programme (FP7/2007-2013) under REA grant agreement n$^\circ$[628803].
His work is also supported in part by Grant 16-01-00367 by RFBR. The work of D.Z.\ is supported in parts by Grant 16-01-00367 by RFBR,
the Program of Fundamental Researches of Russian Academy of Sciences ``Modern Problems of Fundamental Mathematics'', and by Project SFB 1283 of Bielefeld University}

\maketitle

\section{Introduction and main results}\label{sec:main}
\subsection{Introduction}
%\tc{We start with the one-dimensional case.}
Let $\xi_1,\ldots,\xi_n\in\R^1$ be i.i.d.\ random variables with continuous, symmetric distribution, that is, $\P[\xi_1=x] =0$ and $\P[\xi_1< -x] = \P[\xi_1 >x]$ for all $x\in\R$. Consider the one-dimensional \textit{random walk}
$S_i := \xi_1+\ldots+\xi_i$, $1\leq i\leq n$.
We are interested in the random variable
\begin{equation}\label{eq:def_N_n}
N_n = \sum_{i=1}^n \ind_{\{S_i >0\}}
\end{equation}
counting the number of positive terms in the random walk. The classical \textit{discrete arcsine law} due to Sparre Andersen~\cite[Theorem~1 and Eq. (8)]{Sparre0} and~\cite[Theorem~4]{Sparre2} %~\cite[Theorem~C]{sparre_andersen2}
states that
\begin{equation}\label{eq:arcsine_occupation1}
\P[N_n = m] = \frac 1 {2^{2n}} \binom{2m}{m} \binom{2n-2m}{n-m},
\quad m=0,\ldots, n.
\end{equation}
Passing to the limit gives
\begin{equation}\label{2330}
\lim_{n\to\infty} \P\left[\frac{N_n}{n}\leq x \right] = \frac 2 \pi \arcsin \sqrt x, \quad 0\leq x\leq 1,
\end{equation}
which justifies the name of the law in~\eqref{eq:arcsine_occupation1}.

A discussion of this remarkable result, together with a simplified, purely combinatorial proof,  can be found in Feller's book~\cite[Vol II, Section XII.8]{Feller}. Observe that formula~\eqref{eq:arcsine_occupation1} is distribution-free since the values on the right-hand side do not depend on the distribution of $\xi_1$ as long as this distribution is continuous and symmetric.

The aim of the present paper is to obtain a multidimensional generalization of this result to random walks in $\R^d$. Let us start with a special case.
%\tc{Accordingly, from this point on we assume that $S_i= \xi_1+\ldots+\xi_i, 1 \le i \ge n$, is a multidimensional random walk whose increments $\xi_1,\xi_2,\ldots$ are i.i.d. random vectors in $\R^d, d \ge 1$.}
For $m=0$ and $m=n$, formula~\eqref{eq:arcsine_occupation1} provides an expression for the so-called \textit{persistence probability}
\begin{equation}\label{eq:N_n=0}
\P[N_n=0] =
\P[S_1 < 0,\ldots, S_n < 0]
=
\P[S_1 > 0,\ldots, S_n > 0]
=
\P[N_n=n]
= \frac 1 {2^{2n}}\binom {2n}{n}.
\end{equation}
This formula has been generalized to the $d$-dimensional case in the following way~\cite{KVZ15}. Consider a $d$-dimensional random walk $S_i = \xi_1+\ldots+\xi_i$, $1 \le i \le n$, whose increments $\xi_1,\xi_2,\ldots$ are i.i.d.\ random vectors in $\R^d$, $d \in \N$, with centrally symmetric distribution, that is $\xi_1 \eqdistr -\xi_1$. Denote  by $\conv(x_1,\ldots,x_k)$ the \emph{convex hull} of any $k$ points $x_1,\ldots,x_k\in\R^d$, that is
$$
\conv(x_1,\ldots,x_k) = \{\alpha_1 x_1+\ldots+\alpha_k x_k \colon \alpha_1,\ldots,\alpha_k\geq 0, \alpha_1+\ldots+\alpha_k =1\},
$$
and consider the \textit{non-absorption probability}
$\P[0\notin \conv (S_1,\ldots,S_n)]$. If $d=1$, this probability equals $2\P[N_n=0]$ because $0\notin \conv (S_1,\ldots,S_n)$ if and only if either $S_1>0,\ldots,S_n>0$ or $S_1<0,\ldots,S_n<0$, and the probabilities of these two events are equal by the symmetry assumption. For general $d\in\N$, a distribution-free formula for the non-absorption probability has been  obtained in~\cite{KVZ15} and will be recalled in Example~\ref{ex:absorption_probab} below. In the special case $d=1$, this formula reduces to~\eqref{eq:N_n=0}.

%The aim of the present paper is to obtain a $d$-dimensional generalization of~\eqref{eq:arcsine_occupation1}.
When searching for a multidimensional generalization of~\eqref{eq:arcsine_occupation1} for general $0\leq m\leq n$, the basic question is how to define positivity in $\R^d$.  One possible approach is to declare a vector positive if all of its components are positive. This leads to the question on how much time a random walk or a Brownian motion spends in the positive orthant $\R_+^d$; see the work of Bingham and Doney~\cite{bingham_doney} which contains a review on the higher-dimensional analogues of the arcsine law of this type.
In the present paper we choose a different, coordinate-free approach.

The main idea is that
instead of looking at the random variable $N_n$ itself, we shall find an appropriate generalization of its factorial moments. The advantage of working with factorial moments instead of the usual power moments will become evident later.

Observe that we can rewrite~\eqref{eq:arcsine_occupation1} as follows:
\begin{equation}\label{eq:arcsine_occupation2}
\E \left[\binom{N_n}{k}\right]
=
\sum_{m=k}^n
\frac 1 {2^{2m}} \binom{2m}{m} \binom{2n-2m}{n-m} \binom {m}{k}
=
\frac 1 {2^{2k}} \binom {2k}{k} \binom{n}{k},
\quad k=0,1,\ldots,n.
\end{equation}
We omit a direct proof of the second equality because it will be recovered as a special case of Theorem~\ref{theo:arcsine_multidim} presented below. Note that since $N_n$ takes values in $\{0,\ldots,n\}$, the factorial moments in~\eqref{eq:arcsine_occupation2} determine the law of $N_n$ uniquely and therefore statements~\eqref{eq:arcsine_occupation1} and~\eqref{eq:arcsine_occupation2} are indeed equivalent.
In fact,~\eqref{eq:arcsine_occupation2} can be viewed as a system of $n+1$ linear equations in the unknowns $\P[N_n = i]$, $0\leq i\leq n$,  with a non-degenerate upper triangular matrix.

We shall give our multidimensional arcsine law in the form of a $d$-dimensional version of~\eqref{eq:arcsine_occupation2}.
Our main result is in showing that this statement admits an interpretation in terms of an equivalent geometric problem concerning Weyl chambers intersected by a generic linear subspace; see Theorem~\ref{theo:intersect_weyl_chambers}.
This geometric interpretation seems to be new even in the one-dimensional case of the discrete arcsine law while the proofs of the arcsine law given in~\cite{Sparre0}, \cite{Sparre2}, and~\cite{Feller} are purely combinatorial.

Moreover, in the special one-dimensional case there is a different (and new) geometric interpretation of the discrete arcsine law~\eqref{eq:arcsine_occupation1} itself. Let us describe it. Consider the simplex
\[
\{(\beta_1,\ldots,\beta_n) \in \R^n\colon 1 \geq \beta_1 \geq \beta_2 \geq \ldots \geq \beta_n\geq 0\}
\]
or, equivalently, the convex hull of the following $n+1$ points:
\begin{equation*}
(0,0,\ldots,0),\;\; (1,0,\ldots,0),\;\;(1,1,\ldots,0),\;\; \ldots,\;\;(1,1,\ldots,1).
\end{equation*}
There are $2^n n!$ isometric simplices obtained by applying to the above simplex orthogonal transformations of $\R^n$ that permute the coordinates and change their signs. In other words, these simplices are the closed Weyl chambers of type $B_n$ intersected with the cube $[-1,1]^n$. Their union is exactly $[-1,1]^n$, and the interiors of the simplices are disjoint.

Let $H\subset\R^n$ be any generic open
%(to be detailed in Section~\ref{2040} below)
half-space with the boundary passing through the origin. It will be shown that the discrete arcsine probability in~\eqref{eq:arcsine_occupation1} equals the fraction of the above simplices with exactly $m$ vertices lying in $H$; see Corollary~\ref{cor: number of vertices} in Section~\ref{sec:subspaces_intersect_Weyl_chambers} below.
%divided by $2^nn!$ (the number of all simplices).
The asymptotic arcsine law~\eqref{2330} interprets as follows: for any fixed $x\in[0,1]$, the relative fraction of the simplices having at most $x n$ vertices in $H$ tends to $\frac 2 \pi \arcsin \sqrt x$ as $n\to\infty$.

\subsection{Arcsine law for random walks}
We shall give a generalization of~\eqref{eq:arcsine_occupation2} to random walks in $\R^d$.
Let $\xi_1,\ldots,\xi_n$ be random $d$-dimensional vectors. To avoid trivialities,
we always assume that $n\geq d+1$.
The \emph{$d$-dimensional random walk} $(S_i)_{i=1}^n$ with increments $\xi_1,\ldots,\xi_n$ %and the initial point $S_0:=0$
is defined by
$$
S_i := \xi_1+\ldots+\xi_i,\;\; 1\leq i\leq n.
$$
We impose the following assumptions
on the increments $\xi_1,\ldots,\xi_n$:
\begin{itemize}
%\item
\item[$(\pm\text{Ex})$] \textit{Symmetric exchangeability:} For every permutation $\sigma$ of the set $\{1,\ldots,n\}$ and every $\eps_1,\ldots,\eps_n\in \{-1,+1\}$ there is the distributional equality
    $$
    (\xi_1,\ldots,\xi_n) \eqdistr (\eps_1 \xi_{\sigma(1)}, \ldots, \eps_n \xi_{\sigma(n)}).
    $$
\item[$(\text{GP})$] \textit{General position:}
%\footnote{The first arXiv version of~\cite{KVZ15} uses a slightly more restrictive assumption %that the tuple $(\xi_1,\ldots,\xi_n)$ has a joint density on $\R^{nd}$. In fact, the results of~%\cite{KVZ15} hold under the general position assumption, as will be shown in the subsequent %versions of~\cite{KVZ15}. Essentially, all the proofs of~\cite{KVZ15} apply under the general %position assumption if one takes into account Lemmas~\ref{lem:ker_A_gen_pos}, %\ref{lem:ker_A_gen_pos_Br} of the present paper. The same remark applies to random bridges %defined below.}
%The vectors $\xi_{1}, \ldots,\xi_{d}$ are linearly independent with probability $1$.
For every $1\leq i_1 < \ldots < i_d\leq n$, the probability that the vectors $S_{i_1}, \ldots,S_{i_d}$ are linearly dependent, is $0$.
%$(\xi_1,\ldots,\xi_n)$ has a joint Lebesgue density on $\R^{nd}$.
\end{itemize}
\begin{remark}\label{rem:iid}
%These assumptions are satisfied if $\xi_1,\ldots,\xi_n$ are independent identically distributed, $\xi_1$ has the same distribution as $-\xi_1$, and for every affine hyperplane $H\subset \R^d$ we have $\P[\xi_1\in H] = 0$.
If $\xi_1,\xi_2,\ldots$ are independent identically distributed in $\R^d$ and such that $\xi_1$ has the same distribution as $-\xi_1$, then~$(\pm\text{Ex})$ is satisfied and the following conditions are equivalent:
\begin{itemize}
\item [(i)] $(\text{GP})$ holds for all $n\geq d+1$;
\item [(ii)] for every affine hyperplane $H\subset \R^d$ we have $\P[\xi_1\in H] = 0$;
\item [(iii)] for every hyperplane $H_0\subset \R^d$ passing through the origin and every $i\in\N$, we have $\P[S_i \in H_0] =0$.
\end{itemize}
This statement is proved in~\cite[Proposition 2.5]{KVZ15} (which does assume that $\xi_1 \eqdistr - \xi_1$ but does not state this explicitly in the published version).
%will be given in Section~\ref{sec:general_position}.
%
%and $\xi_1$ is absolutely continuous with a density $f$ which is centrally symmetric with respect to the origin, that is
%$$
%f(t) = f(-t), \quad t\in\R^d.
%$$
\end{remark}

%Consider first the one-dimensional case $d=1$.
%counting the number of positive points of the random walk. A classical result of Sparre Andersen~\cite[Theorem~C]{sparre_andersen2} states that
%Given a random walk $(S_i)_{i=0}^n$ in $\R^d$, define a random variable
For $1\leq k\leq n$ we are interested in the random variable equal to half the number of polytopes of the form $\conv (S_{i_1},\ldots, S_{i_k})$ that do not contain the origin:
\begin{equation}\label{eq:def_M_n_k}
M_{n,k}^{(d)} = \frac{1}{2} \sum_{1\leq i_1 < \ldots < i_k\leq n} \ind_{\{0\notin \conv(S_{i_1},\ldots, S_{i_k})\}}.
\end{equation}

In the one-dimensional case $d=1$ the convex hull of $S_{i_1},\ldots, S_{i_k}$ does not contain the origin if and only if the numbers $S_{i_1},\ldots, S_{i_k}$ have the same sign, whence
$$
M_{n,k}^{(1)} = \frac 12 \binom {N_n}{k} + \frac 12 \binom {n-N_n}{k} \quad \text{a.s.}
$$
Here we used that $\P[S_{i} = 0] = 0$, $1\leq i \leq n$, which holds by the general position assumption $(\text{GP})$. Using the fact that  $N_n$ has the same distribution as $n-N_n$, which is a consequence of assumption $(\pm\text{Ex})$, we deduce that
\begin{equation} \label{2327}
\E M_{n,k}^{(1)} = \E \left[\binom{N_n}{k}\right].
\end{equation}
Therefore, we can view $\E M_{n,k}^{(d)}$ as a $d$-dimensional generalization of $\E\left[\binom{N_n}{k}\right]$.
Our main result generalizes~\eqref{eq:arcsine_occupation2} to arbitrary dimension as follows.
%The next theorem is a multidimensional generalization of~\eqref{eq:arcsine_occupation2}.
\begin{theorem}\label{theo:arcsine_multidim}
Consider a random walk $(S_i)_{i=1}^n$ in $\R^d$, $n\geq d+1$, with increments $\xi_1,\ldots,\xi_n$ satisfying assumptions $(\pm\text{Ex})$ and $(\text{GP})$. For every $k=1,\ldots,n$ we have
\begin{equation}\label{eq:E_M_n_k}
\E M_{n,k}^{(d)}  =  \binom n k \frac {B(k, d-1) + B(k, d-3) +\ldots} {2^k k!} = \binom n k \E M_{k,k}^{(d)},
\end{equation}
where the $B(k,j)$'s are defined by their generating function
\begin{equation}\label{eq:def_B_k_j}
(t+1) (t+3) \ldots (t+2k-1) = \sum_{j=0}^{k} B(k,j) t^j.
\end{equation}
We put $B(k,j) = 0$ for $j<0$ and $j>k$ so that the sum in~\eqref{eq:E_M_n_k} has only finitely many non-zero terms.
\end{theorem}
\begin{example}
For $d=1$ Theorem~\ref{theo:arcsine_multidim} reduces to~\eqref{eq:arcsine_occupation2} in view of $B(k,0) = (2k-1)!!$:
$$
\E \left[\binom{N_n}{k}\right] = \E M_{n,k}^{(1)}  = \binom n k \frac {(2k-1)!!} {2^k k!}  =\frac 1 {2^{2k}} \binom {2k}{k} \binom{n}{k}.
$$
Note in passing that this proves the second equality in~\eqref{eq:arcsine_occupation2}.
\end{example}
\begin{example}\label{ex:absorption_probab}
In the case $k=n$ Theorem~\ref{theo:arcsine_multidim} provides a formula for the non-absorption probability
\begin{equation}\label{eq:absorption_B_n}
\P[0\notin \conv(S_1,\ldots,S_n)]
=
\frac {2(B(n, d-1) + B(n, d-3) +\ldots)} {2^n n!}.
\end{equation}
This formula was obtained in~\cite{KVZ15}.
\end{example}
\begin{remark}
For the number of polytopes of the form $\conv (S_{i_1},\ldots,S_{i_k})$ \textit{containing} the origin we have the formula
\begin{equation}\label{eq:expected_contain_origin}
\E \left[\sum_{1\leq i_1 < \ldots < i_k\leq n} \ind_{\{0\in \conv(S_{i_1},\ldots, S_{i_k})\}}\right]
=
2 \binom n k \frac {B(k, d+1) + B(k, d+3) +\ldots} {2^k k!}
\end{equation}
which follows from~\eqref{eq:E_M_n_k} and the identities
$$
B(k,2)+B(k,4)+\ldots = B(k,1) + B(k,3)+\ldots = 2^{k-1}k!.
$$
To prove these identities, take $t=\pm 1$ in~\eqref{eq:def_B_k_j}. Note that for $1\leq k \leq d$ both sides of~\eqref{eq:expected_contain_origin} vanish (the left hand-side vanishes by assumption $(\text{GP})$).
\end{remark}

Let us now pass to the large $n$ limit. The classical arcsine law~\cite{erdoes_kac,Sparre0} for the number of positive terms in a one-dimensional random walk (whose increments are symmetrically distributed or have zero mean and finite positive variance, the case not considered in our paper) can be stated as follows:
\begin{equation} \label{1547}
\lim_{n\to\infty} \P\left[\frac{N_n}{n}\leq x \right] = \frac 2 \pi \arcsin \sqrt x, \quad 0\leq x\leq 1.
\end{equation}
In terms of moments (which fully define any distribution concentrated on $[0,1]$), this can equivalently be written as
\begin{equation} \label{2344}
\lim_{n\to\infty} \E \left[\frac{N_n^k}{n^k}\right] = \frac 1 {2^{2k}} \binom {2k}{k}, \quad k\in\N.
\end{equation}
Note that
\begin{equation} \label{eq: N_n moments}
\E N_n^k=k!\E \left[\binom{N_n}{k}\right](1+o(1)),\quad n\to\infty,
\end{equation}
which together with~\eqref{2327} implies that~\eqref{2344} is equivalent to
\[
\lim_{n\to\infty} \E \left[\frac{k! M_{n,k}^{(1)}}{n^k}\right] =\frac 1 {2^{2k}} \binom {2k}{k}, \quad k\in\N.
\]
From Theorem~\ref{theo:arcsine_multidim} we obtain the following $d$-dimensional generalization of \eqref{eq: N_n moments}:
\begin{equation} \label{eq:limit E_M_n_k}
\lim_{n\to\infty} \E \left[\frac{k! M_{n,k}^{(d)}}{n^k}\right] = \frac {B(k,d-1) + B(k, d-3)+\ldots}{2^k k!} = \E M_{k,k}^{(d)}, \quad k\in\N.
\end{equation}
%\begin{remark}
%\tc{Thus, in the $d$-dimensional setting, the sequence in \eqref{eq:limit E_M_n_k} is analogous to the sequence of moments of the %%arcsine distribution. However, unlike the one-dimensional case,} this sequence does not correspond to a distribution on $[0,1]$ %because  both the first and the second moments of such distribution must be $\frac12$.

%One may ask if such distribution exists for a modified  sequence with different first $d$ values (which are prescribed to be $%\frac12$ according to our natural definition of $M_{n,k}^{(d)}$ but should not  have any interpretation in terms of the distribution %we try to find.) The answer is negative, at least for $d \le 10$: it is possible to check by hand that the sequence  in the right-%hand side of~\eqref{eq:limit E_M_n_k} is not completely monotonic and thus it is not the sequence of moments of a distribution on %$[0,1]$.
%\end{remark}

\begin{remark}\label{1315}
Thus, in the $d$-dimensional setting, the sequence in \eqref{eq:limit E_M_n_k} is analogous to the sequence of moments of the arcsine distribution. However, unlike the one-dimensional case, for $d \ge 2$ this sequence does not correspond to a distribution on $[0,1]$ because  both the first and the second moments of such hypothetical distribution should be $\frac12$. On the other hand, it allows the following natural interpretation.

Let $W^{(d)}(t)$, $t\in [0,1]$, be a standard Brownian motion taking values in $\R^d$. Consider the random variable
$$
M_{\infty, k}^{(d)} := \frac 12 \int_{0< t_1<\ldots < t_k < 1} \ind_{\{ 0\notin \conv(W^{(d)}(t_1),\ldots, W^{(d)}(t_k)) \}} \dd t_1 \dots \dd t_k.
$$
Note that if $U_1,\ldots,U_k$ are i.i.d.\ random variables distributed uniformly on $[0,1]$ and independent of $W^{(d)}$, then
$$
\E M_{\infty, k}^{(d)} = \frac{1}{2\cdot k!} \P[0\notin \conv(W^{(d)}(U_1),\ldots, W^{(d)}(U_k))].
$$
Let $k$ and $d$ be fixed, while $n\to\infty$. Using Donsker's invariance principle, it is possible to show that  $\frac{1}{n^k} M_{n,k}^{(d)}$ converges weakly (together with all moments) to $M_{\infty, k}^{(d)}$.
Hence, by \eqref{eq:limit E_M_n_k},
\begin{multline} \label{eq: Bm interpretation}
\lim_{n\to\infty} \E \left[\frac{k! M_{n,k}^{(d)}}{n^k}\right] = k! \E M_{\infty,k}^{(d)} = \E M_{k,k}^{(d)} = \frac{1}{2}\P[0\notin\conv(W^{(d)}(U_1),\dots,W^{(d)}(U_k))].
\end{multline}

The same result may be obtained by observing that the increments of the sequence $W^{(d)}(U_{(1)}),\ldots, W^{(d)}(U_{(k)})$, where $U_{(1)},\ldots,U_{(k)}$ are the order statistics of $U_1,\ldots,U_k$, are exchangeable, and applying Example 1.4 directly. This explains why $\E M_{k,k}^{(d)}$ appears in \eqref{eq: Bm interpretation}. We also note that the equality $\E M_{k,k}^{(1)} = k! \E M_{\infty,k}^{(1)}$ implies directly that $\E M_{k,k}^{(1)}$ is the $k$-th moment of the arcsine law. This is easily seen from the definition of $M_{\infty, k}^{(1)}$ and the following arcsine law for the Brownian motion:
$$\P \left [\int_0^1 \ind_{\{ W^{(1)}(t)>0 \}} \dd t \le x \right] = \frac 2 \pi \arcsin \sqrt x, \quad 0\leq x\leq 1.$$
\end{remark}

\begin{remark}
Theorem~\ref{theo:arcsine_multidim} shows that the \emph{expectation} of $M_{n,k}^{(d)}$ does not depend on the distribution of increments of the random walk. One may ask whether the \emph{distribution} of $M_{n,k}^{(d)}$ has the same property. Our simulations gave a strong evidence against this conjecture.
\end{remark}

We shall give two proofs of Theorem~\ref{theo:arcsine_multidim}.
The first proof deduces Theorem~\ref{theo:arcsine_multidim} from a geometric result on the number of $k$-faces of a random Weyl chamber that are intersected by a linear subspace. This result, which is of an independent interest, will be stated in Theorem~\ref{theo:intersect_weyl_chambers}, Section~\ref{sec:subspaces_intersect_Weyl_chambers}. The second proof, given in Section~\ref{sec:proofs_B_n}, is based on~\eqref{eq:absorption_B_n}.
Both proofs strongly rely on the ideas and results of~\cite{KVZ15}. The second proof is shorter but it does not allow any interpretation of Theorem~\ref{theo:arcsine_multidim}. In the next section, a similar statement for random bridges  is formulated in Theorem~\ref{theo:arcsine_multidim_Br} (which will be proved in Section~\ref{017}).

\subsection{Uniform law for random bridges}
Similar results can be obtained for random bridges which are essentially random walks required to return to the origin after $n$ steps.
Let $\xi_1,\ldots,\xi_n$ be random vectors in $\R^d$, where the reader may always assume that $n\geq d+2$ to avoid trivialities.
We define the partial sums $(S_i)_{i=1}^n$  by
$$
S_i:= \xi_1+\ldots+\xi_i, \quad 1\leq i\leq n, %\quad S_0:=0,
$$
and impose the following assumptions on the increments $\xi_1,\ldots,\xi_n$:
\begin{itemize}
\item[$(\text{Br})$] \textit{Bridge property:} $S_n=\xi_1+\ldots+\xi_n = 0$ a.s.
\item[$(\text{Ex})$] \textit{Exchangeability:} For every permutation $\sigma$ of the set $\{1,\ldots,n\}$ we have the distributional equality
$$
(\xi_{\sigma(1)},\ldots,  \xi_{\sigma(n)}) \eqdistr (\xi_1,\ldots,\xi_n).
$$
\item[$(\text{GP}')$] \textit{General position:}
%The vectors $\xi_{1}, \ldots, \xi_{d}$ are linearly independent with probability $1$.
For every $1\leq i_1 < \ldots < i_d \leq n-1$, the probability that the vectors $S_{i_1}, \ldots, S_{i_d}$ are linearly dependent, is $0$.
%\textit{Restricted absolute continuity:} $(\xi_1,\ldots,\xi_{n-1})$ has a joint Lebesgue density on $\R^{(n-1)d}$.
\end{itemize}
The stochastic process $(S_i)_{i=1}^n$ is called a \emph{random bridge}. Note  that assumption  $(\text{Ex})$ does not require invariance with respect to sign changes and thus is weaker  than $(\pm\text{Ex})$.

For $d=1$ the distribution of the random variable $N_n$ counting the number of positive terms among $S_1,\ldots,S_{n-1}$ is discrete uniform on $\{0,\ldots,n-1\}$ by a result of Sparre Andersen~\cite[Corollary~2]{Sparre1953}. That is,
\begin{equation}\label{eq:arcsine_occupationBr1}
\P[N_n = m] = \frac 1n, \quad m=0,\ldots,n-1.
\end{equation}
Alternatively, this formula follows easily from~\cite[Theorem~2.1]{spitzer}. In terms of factorial moments, \eqref{eq:arcsine_occupationBr1} can be stated as follows:
\begin{equation}\label{eq:arcsine_occupationBr2}
\E \left[\binom{N_n}{k}\right] = \frac 1n \sum_{m=k}^{n-1} \binom m {k} =  \frac 1{k+1} \binom{n-1}{k}, \quad k=0,\ldots, n-1.
\end{equation}
The second equality in~\eqref{eq:arcsine_occupationBr2} follows easily by induction over $n$.
Note that~\eqref{eq:arcsine_occupationBr1} and~\eqref{eq:arcsine_occupationBr2} are equivalent similarly to the case of random walks we seen above.
%(namely, the moments determine the distribution uniquely) because $N_n$ takes values in $%\{0,\ldots,n-1\}$.

To state a $d$-dimensional generalization of~\eqref{eq:arcsine_occupationBr2}, we consider a slight modification of $M_{n,k}^{(d)}$, namely
\begin{equation}\label{eq:def_M_n_k_Br}
M_{n,k}^{(d)} = \frac{1}{2} \sum_{1\leq i_1 < \ldots < i_{k}\leq n-1} \ind_{\{0\notin \conv(S_{i_1},\ldots, S_{i_{k}})\}}.
\end{equation}
We excluded the case $i_k=n$ because the corresponding convex hulls would contain $0$ by the assumption $S_n=0$ a.s.
%Note a slight difference to~\eqref{eq:def_M_n_k}.
The following is our  main result for random bridges.
\begin{theorem}\label{theo:arcsine_multidim_Br}
Consider a random bridge $(S_i)_{i=1}^n$ in $\R^d$, $n\geq d+2$, whose increments $\xi_1,\ldots,\xi_n$ satisfy assumptions $(\text{Br})$, $(\text{Ex})$, $(\text{GP}')$.
For all $k=1,\ldots,n-1$ we have
\begin{equation}\label{eq:theo:arcsine_multidim_Br}
\E M_{n,k}^{(d)} = \frac 1 {(k+1)!} \binom {n-1} {k} \left(\stirling{k+1}{d} + \stirling{k+1}{d-2} +\ldots\right) =  \binom {n-1} {k} \E M_{k+1, k}^{(d)} ,
\end{equation}
where $\stirling {k+1}{1}, \ldots, \stirling{k+1}{k+1}$ are the Stirling numbers of the first kind defined by the formula
\begin{equation}
t(t+1)\ldots (t+k) = \sum_{j=1}^{k+1} \stirling {k+1}{j}t^j.
\end{equation}
We use the convention $\stirling{k+1}{j} = 0$ for $j<0$ and $j>k+1$, so the sum in~\eqref{eq:theo:arcsine_multidim_Br} contains a finite number of non-zero terms.
\end{theorem}
\begin{example}
In the one-dimensional case $d=1$ we have
$$
M_{n,k}^{(1)} = \frac 12 \binom {N_n}{k} + \frac 12 \binom {n-N_n-1}{k} \quad \text{a.s.}
$$
%$$
%\E M_{n,k}^{(1)} = \E \left[\binom{N_n}{k}\right],
%$$
%so that $\E M_{n,k}^{(d)}$ can be viewed as a $d$-dimensional generalization of $\E \left[\binom{N_n}{k}\right]$.
Theorem~\ref{theo:arcsine_multidim_Br} yields
$$
\frac 12 \E \left[\binom {N_n}{k} + \binom {n-N_n-1}{k}\right]
=
\E M_{n,k}^{(1)} = \frac 1 {(k+1)!} \binom {n-1} {k} \stirling{k+1}{1} = \frac 1 {k+1} \binom {n-1} {k}
$$
in view of $\stirling{k+1}{1} = k!$. Then  we recover~\eqref{eq:arcsine_occupationBr2} using the
distributional equality $N_n \eqdistr n-1-N_n$, which itself follows by $(S_i)_{i=1}^n \eqdistr (S_n - S_{n-i})_{i=1}^n \stackrel{\text{a.s.}}{=} -( S_{n-i})_{i=1}^n$, a consequence of $(\text{Ex})$ and $(\text{Br})$.

\end{example}
\begin{example}
In the case $k=n-1$ Theorem~\ref{theo:arcsine_multidim_Br} reduces to the formula for the non-absorption probability
$$
\P[0\notin \conv(S_1,\ldots,S_{n-1})]
=
\frac 2 {n!} \left(\stirling{n}{d} + \stirling{n}{d-2} +\ldots\right)
$$
which was obtained in~\cite{KVZ15}.
\end{example}
As $n\to\infty$, the random variable $N_n/n$ converges weakly to the uniform distribution on the interval $[0,1]$, namely
$$
\lim_{n\to\infty} \P\left[\frac{N_n}{n}\leq x\right] = x, \quad x\in [0,1].
$$
In terms of moments, we can write this as
$$
\lim_{n\to\infty} \E \left[\frac{N_n^k}{n^k}\right] = \frac 1{k+1}, \quad k\in\N,
$$
which is a bridge analogue of~\eqref{2344} and, similarly, equivalent to
\[
\lim_{n\to\infty} \E \left[\frac{k! M_{n,k}^{(1)}}{n^k}\right] =\frac 1{k+1}, \quad k\in\N.
\]
Theorem~\ref{theo:arcsine_multidim_Br} yields the following $d$-dimensional version of this relation:
$$
\lim_{n\to\infty} \E \left[\frac{k! M_{n,k}^{(d)}}{n^k}\right] = \frac 1 {(k+1)!}\left(\stirling{k+1}{d} + \stirling{k+1}{d-2} +\ldots\right) = \E M_{k+1,k}^{(d)}, \quad k\in\N.
$$
Finally, there is the following analogue of \eqref{eq: Bm interpretation}: if $W_0^{(d)}(t)$, $t\in [0,1]$, is a Brownian bridge in $\R^d$ and
$U_1,\ldots,U_k$ are i.i.d.\ random variables distributed uniformly on $[0,1]$ and independent of $W_0^{(d)}$, then
$$
\lim_{n\to\infty} \E \left[\frac{k! M_{n,k}^{(d)}}{n^k}\right] = \E M_{k+1,k}^{(d)} = \frac{1}{2}\P[0\notin\conv(W_0^{(d)}(U_1),\dots,W_0^{(d)}(U_k))].
$$

\section{Relation to linear subspaces intersecting Weyl chambers}\label{sec:subspaces_intersect_Weyl_chambers}
In this section we prove Theorems~\ref{theo:arcsine_multidim} and~\ref{theo:arcsine_multidim_Br} by deducing them from certain geometric results on linear subspaces intersecting Weyl chambers of types $B_n$ and $A_{n-1}$, respectively. We start by recalling the necessary definitions.

\subsection{The reflection group  and Weyl chambers of type $B_n$}\label{2040}
The \emph{reflection group} $\mathcal G(B_n)$ of type $B_n$ acts on $\R^n$ by permuting the coordinates in an arbitrary way and by multiplying any number of coordinates by $-1$. That is, the elements of $\mathcal G(B_n)$ are isometries of the form
$$
g_{\eps,\sigma}:\R^n\to\R^n, \quad (\beta_1,\ldots,\beta_n) \mapsto (\eps_1 \beta_{\sigma(1)},\ldots,\eps_n \beta_{\sigma(n)}),
$$
where $\sigma\in \Sym(n)$ is a permutation of the set $\{1,\ldots,n\}$ and $\eps=(\eps_1,\ldots,\eps_n)\in \{-1,+1\}^n$. Here we denote by $\Sym (n)$ the symmetric group on the set $\{1,\ldots,n\}$.
The group $\mathcal G(B_n)$ is the symmetry group of the $n$-dimensional cube $[-1,1]^n$ and the number of elements in this group is $2^n n!$.

A set $Q\subset \R^n$ is called a \emph{convex cone} if for all $x,x'\in Q$ and $\alpha,\alpha'>0$ we have $\alpha x +\alpha' x' \in Q$. We refer to~\cite{ALMT14,amelunxen_lotz} and~\cite[Section~6.5]{SW08} for information on convex cones and spherical convex geometry. We shall consider only \emph{polyhedral cones}. These are defined as finite intersections of half-spaces whose boundaries pass through the origin.  The faces of the cone are obtained by replacing in the above definition some of the half-spaces by their boundaries and taking the intersection.
The \emph{fundamental Weyl chamber of type $B_n$} is the convex cone given by
$$
\mathcal C(B_n) = \{(\beta_1,\ldots,\beta_n) \in \R^n\colon 0 < \beta_1 < \beta_2 < \ldots < \beta_n\}.
$$
This is a fundamental domain for $\mathcal G(B_n)$, meaning that the cones $g \, \mathcal C(B_n)$, $g\in \mathcal G(B_n)$,  are disjoint and their closures (which will be called \emph{closed Weyl chambers} or, without any risk of confusion, simply \emph{Weyl chambers}) cover the whole $\R^n$. Thus, the closed Weyl chambers are the convex cones given by
$$
C_{\eps,\sigma}^B := \{(\beta_1,\ldots,\beta_n)\in\R^n \colon \eps_1 \beta_{\sigma(1)} \geq  \eps_2 \beta_{\sigma(2)} \geq \ldots\geq \eps_n \beta_{\sigma(n)} \geq 0\},
$$
where $\eps\in \{-1,+1\}^n$, $\sigma\in \Sym(n)$. The terms are required to be non-increasing rather than increasing for convenience of proofs, and the superscript $B$ refers to the type of the chambers.
We denote by $\mathcal F_k(Q)$ the set of all (closed) $k$-dimensional faces of a convex cone $Q$. For $1\leq k\leq n$, the $k$-dimensional faces of $C_{\eps,\sigma}^B$ are indexed by collections $1\leq i_1<\ldots < i_k\leq n$ and have the form
\begin{multline}\label{eq:C_eps_sigma_face}
C_{\eps,\sigma}^B(i_1,\ldots,i_k) :=
\{(\beta_1,\ldots,\beta_n)\in\R^n \colon
\eps_1 \beta_{\sigma(1)} = \ldots = \eps_{i_1}\beta_{\sigma(i_1)}
\\\geq
\eps_{i_1+1} \beta_{\sigma(i_1+1)} =\ldots = \eps_{i_2}\beta_{\sigma(i_2)}
\geq
\ldots
\geq
\eps_{i_{k-1}+1} \beta_{\sigma(i_{k-1}+1)} = \ldots =\eps_{i_k} \beta_{\sigma(i_k)}
\\\geq
\beta_{\sigma(i_k+1)} =\ldots = \beta_{\sigma(n)} =0\}.
\end{multline}
In the case $i_k=n$, no $\beta_i$'s are required to be $0$. In particular, $\# \mathcal F_k(C_{\eps,\sigma}^B) = \binom nk$.

A \textit{hyperplane arrangement} is a finite set of (distinct) hyperplanes in $\R^n$.
The \emph{reflection arrangement $\mathcal A(B_n)$ of type $B_n$} consists of the hyperplanes
\begin{equation}\label{eq:reflection_arrangement}
\{\beta_ i = \beta_j\} \;\; (1\leq i < j\leq n),
\quad
\{\beta_ i = - \beta_j\} \;\; (1\leq i < j\leq n),
\quad
\{\beta_i = 0\} \;\; (1\leq i\leq n).
\end{equation}
The name is due to the fact that  reflections with respect to these hyperplanes generate the group $\mathcal G(B_n)$.
The \emph{lattice} $\mathcal L(B_n)$ generated by the reflection arrangement of type $B_n$ consists of linear subspaces of $\R^n$ which can be represented as intersections of the hyperplanes~\eqref{eq:reflection_arrangement}. We say that a linear subspace $L\subset \R^n$ of codimension $d$ is \emph{in general position} with respect to the reflection arrangement if for every linear subspace $K\in \mathcal L(B_n)$ we have
\begin{equation}\label{eq:def_gen_pos}
\dim (L\cap K) =
\begin{cases}
\dim K - d, &\text{if } \dim K \geq d,\\
0, &\text{if } \dim K \leq d.
\end{cases}
\end{equation}

\subsection{Subspaces intersecting faces of Weyl chambers of type $B_n$}
The next theorem, which is the main result of the present section, will be shown to imply Theorem~\ref{theo:arcsine_multidim}.

\begin{theorem}\label{theo:intersect_weyl_chambers}
Let $L\subset \R^n$ be a deterministic linear subspace of codimension $d$ in general position with respect to the reflection arrangement~\eqref{eq:reflection_arrangement} of type $B_n$. Let $Q$ be sampled randomly and uniformly among the $2^n n!$ closed Weyl chambers $C_{\eps,\sigma}^B$ of type $B_n$. Then the expected number of $k$-dimensional faces of $Q$ intersected  by $L$ in a trivial way is given by
\begin{align*}
\E \left[\sum_{F\in \mathcal F_k(Q)} \ind_{\{F\cap L = \{0\}\}}\right]
&\eqdef
\frac{1}{2^n n!} \sum_{\eps\in \{-1,+1\}^n} \sum_{\sigma\in \Sym(n)}
\sum_{F\in \mathcal F_k(C_{\eps,\sigma}^B)} \ind_{\{F\cap L = \{0\}\}}\\
&=
 2 \binom n k \frac {B(k, d-1) + B(k, d-3) +\ldots} {2^k k!},
\end{align*}
where the $B(k,j)$'s are defined in Theorem~\ref{theo:arcsine_multidim}.
\end{theorem}
Here, we say that $Q$ and $L$ intersect {\it in a trivial way} if $Q \cap L = \{0\}$.

\begin{corollary} \label{cor: number of vertices}
Let $d=1$ (so $L$ is a hyperplane) and let $H$ be either of the open half-spaces with the boundary $L$. Then the number of vertices $V_n$ of the random simplex $Q \cap [-1,1]^n$ lying in $H$ follows the discrete arcsine distribution~\eqref{eq:arcsine_occupation1}.
\end{corollary}

\begin{proof}[Proof of Corollary~\ref{cor: number of vertices}]
%Consider a closed Weyl chamber $C_{\eps,\sigma}^B$ and fix some $0 \le m \le n$ and $1 \le k \le n$. If the simplex $C_{\eps,\sigma}^B \cap [-1,1]^n$ has either $m$ or $n-m$ vertices in $H$, then the number of the $k$-dimensional faces of $C_{\eps,\sigma}^B$ intersected by $L$ in a trivial way is
%$$
%\sum_{F\in \mathcal F_k(C_{\eps,\sigma}^B)} \ind_{\{F\cap L = \{0\}\}} = \binom m k + \binom{n-m}{k}.
%$$ 
%It is not hard to check that this is actually an equivalence, that is the number of such faces uniquely defines the numbers of vertices to the different sides of $L$.
The number of the $k$-dimensional faces of $Q$ intersected by $L$ in a trivial way equals $\binom{V_n}{k} + \binom{n-V_n}{k}$. By taking the expectations and using the distributional identity $V_n \eqdistr n - V_n$ and the result of Theorem~\ref{theo:intersect_weyl_chambers}, we conclude that the random variables $V_n$ and $N_n$ have the same factorial moments given in~\eqref{eq:N_n=0}. As we argued in the Introduction, these are factorial moments of the discrete arcsine distribution, which is the unique distribution on $\{0, \ldots, n\}$ with the given factorial moments. Therefore $V_n$ and $N_n$ have the same distribution.
\end{proof}

%In particular, the case $d=1$ corresponds to the classical discrete arcsine law~\eqref{eq:arcsine_occupation1} and~%\eqref{eq:arcsine_occupation2}. The proofs of this law given in~\cite[Theorem~C]{sparre_andersen2} and~\cite[Vol II, Section XII.8]%{Feller} are purely combinatorial. The geometric interpretation of the discrete arcsine law given in Theorem~%\ref{theo:intersect_weyl_chambers} seems to be new even in the one-dimensional case.

\subsection{Proof of Theorem~\ref{theo:arcsine_multidim} given Theorem~\ref{theo:intersect_weyl_chambers}}
We shall need a short notation for one of the closed Weyl chambers:
$$
C^B := \{(\beta_1,\ldots,\beta_n) \in \R^n\colon \beta_1 \geq \beta_{2} \geq \ldots \geq \beta_n \geq 0\}.
$$
The next lemma records the relation between random walks and linear subspaces intersecting Weyl chambers. The case $k=n$ of this lemma appeared in~\cite{KVZ15}.
\begin{lemma}\label{lem:geometric_interpretation}
Let $x_1,\ldots,x_n\in\R^d$ be arbitrary vectors and denote by $s_i = x_1+\ldots+x_i$, $1\leq i\leq n$, their partial sums. Let $A: \R^n \to \R^d$ be the linear operator defined on the standard basis $e_1,\ldots,e_n$ of $\R^n$ by $A e_1 =x_1, \ldots, A e_n = x_n$.
Then, the number of collections of indices $1\leq i_1<\ldots < i_k\leq n$ such that $0\in \conv(s_{i_1},\ldots,s_{i_k})$ is equal to the number of $k$-dimensional faces $F$ in the convex cone $C^B$ intersected non-trivially by the linear subspace $\Ker A$.
%(meaning that $F \cap \Ker A \neq \{0\}$).
\end{lemma}
\begin{proof}
For a given collection of indices $1\leq i_1<\ldots < i_k\leq n$ we have $0\in \conv(s_{i_1},\ldots,s_{i_k})$ if and only if there exist $\alpha_1,\ldots,\alpha_k\geq 0$ (not all of them being $0$) such that $\alpha_1 s_{i_1} + \ldots+ \alpha_k s_{i_k} = 0$, or, equivalently,
$$
\alpha_1 (x_1+\ldots+x_{i_1}) + \alpha_2 (x_1+\ldots+x_{i_2}) + \ldots + \alpha_k (x_1+\ldots+x_{i_k}) = 0.
$$
Rearranging the terms, we rewrite this condition as
\begin{multline}\label{eq:lin_comb}
x_1(\alpha_1 + \ldots +\alpha_k) + \ldots + x_{i_1}(\alpha_1 + \ldots +\alpha_k)
\\
+x_{i_1+1} (\alpha_2 + \ldots +\alpha_k) + \ldots + x_{i_2} (\alpha_2+\ldots+\alpha_k)
\\
+\ldots
+x_{i_{k-1}+1} \alpha_k + \ldots + x_{i_k} \alpha_k = 0.
\end{multline}
Introducing the new variables $\beta_1,\ldots,\beta_n$ as the coefficients of $x_1,\ldots,x_n$, that is
\begin{align*}
&\beta_1 =\ldots = \beta_{i_1} := \alpha_1+\ldots+\alpha_k,\\
&\beta_{i_1+1} =\ldots = \beta_{i_2} := \alpha_2+\ldots+\alpha_k,\\
&\ldots,\\
&\beta_{i_{k-1}+1} =\ldots = \beta_{i_k} := \alpha_k,\\
&\beta_{i_k + 1} = \ldots = \beta_n :=0,
\end{align*}
we can rewrite~\eqref{eq:lin_comb} as
$
\beta_1 x_1 + \ldots +\beta _n x_n = 0
$
or, equivalently, $(\beta_1,\ldots,\beta_n)\in \Ker A$.
Our conditions on the $\alpha_i$'s translate into the following equivalent condition on the $\beta_i$'s:
\begin{equation}\label{eq:face_of_Weyl_chamber}
\beta_1 = \ldots = \beta_{i_1} \geq \beta_{i_1+1} =\ldots = \beta_{i_2} \geq \ldots \geq \beta_{i_{k-1}+1} = \ldots =\beta_{i_k} \geq \beta_{i_k+1} =\ldots = \beta_n =0,
\end{equation}
where at least one inequality should be strict, that is $(\beta_1,\ldots,\beta_n)\neq 0$. If $i_k=n$, then there are no $\beta_i$'s required to vanish. Summarizing, we have $0\in \conv(s_{i_1},\ldots,s_{i_k})$ if and only if $F\cap \Ker A \neq \{0\}$, where $F\subset \R^n$ is the closed convex cone defined by~\eqref{eq:face_of_Weyl_chamber}. Since any such $F$ is a $k$-dimensional face of the convex cone $C^B$ and, conversely, any $k$-face has this form, we obtain the required statement.
\end{proof}

\begin{proof}[Proof of Theorem~\ref{theo:arcsine_multidim} given Theorem~\ref{theo:intersect_weyl_chambers}]
Let $A: \R^n \to \R^d$ be the random linear operator defined on the standard basis $e_1,\ldots,e_n$ of $\R^n$ by $A e_1 =\xi_1, \ldots, A e_n = \xi_n$.  By Lemma~\ref{lem:geometric_interpretation},
\begin{equation} \label{eq:M_n,k = distr 1}
M_{n,k}^{(d)} =
\frac{1}{2} \sum_{1\leq i_1 < \ldots < i_k\leq n} \ind_{\{0\notin \conv(S_{i_1},\ldots, S_{i_k})\}}
=
\frac 12 \sum_{F\in \mathcal F_k(C^B)} \ind_{\{F \cap \Ker A = \{0\}\}}
\quad   \text{a.s.},
\end{equation}
where we used that both sums have the same number of terms $\#\mathcal F_{k}(C^B) = \binom n k$. Let us show that together with symmetric exchangeability assumption $(\pm\text{Ex})$, this implies that for every closed Weyl chamber $C_{\eps,\sigma}^B$,
\begin{equation} \label{eq:M_n,k = distr 2}
M_{n,k}^{(d)} \eqdistr \frac 12 \sum_{F\in \mathcal F_k(C_{\eps,\sigma}^B)} \ind_{\{F\cap \Ker A = \{0\}\}}.
\end{equation}

First note that $C_{\eps,\sigma}^B = g (C^B)$ for $g:=g_{\eps((\bar \sigma)^{-1}), (\bar \sigma)^{-1}}$, where the permutation $\bar \sigma$ is defined by $\bar \sigma =(\sigma(n), \dots, \sigma(1))$. This holds by the fact that $g e_k= \eps_{\bar \sigma(k)} e_{\bar \sigma(k)}$ (where $1 \le k \le n$), which ensures that $\bar \sigma$ arranges the absolute values of coordinates of points in $g (C^B)$ in an increasing order, so $\sigma$ arranges them in a decreasing order as needed for $C_{\eps,\sigma}^B$. Further, for any $k$-face $F\in \mathcal F_k(C^B)$,
$$\bigl \{g(F) \cap \Ker A = \{0\} \bigr \} = \bigl \{F \cap g^{-1}(\Ker A) = \{0\} \bigr \} = \bigl \{F \cap \Ker (A g )= \{0\} \bigr \}.$$ Since the random linear operator $Ag$ satisfies $(Ag) e_1 =\eps_{\sigma(n)} e_{\sigma(n)} ,\dots, (Ag) e_n = \eps_{\sigma(1)} e_{\sigma(1)}$, from $(\pm\text{Ex})$ it follows that $\Ker (A g ) \eqdistr \Ker A$. Hence \eqref{eq:M_n,k = distr 2} follows from \eqref{eq:M_n,k = distr 1} as required.

Taking the expectation in \eqref{eq:M_n,k = distr 2} and then the mean over all $2^n n!$ pairs $(\eps,\sigma)$, we obtain
\begin{align*}
2\E M_{n,k}^{(d)}
&=
\E \left[\frac{1}{2^n n!} \sum_{\eps\in \{-1,+1\}^n} \sum_{\sigma\in \Sym(n)}
\sum_{F\in \mathcal F_k(C_{\eps,\sigma}^B)} \ind_{\{F\cap \Ker A = \{0\}\}}\right] \\
&=
2\binom n k \frac {B(k, d-1) + B(k, d-3) +\ldots} {2^k k!},
\end{align*}
where the second equality is by Theorem~\ref{theo:intersect_weyl_chambers} applied to $L:=\Ker A$. The fact that with probability one, $\Ker A$ has codimension $d$ and is in general position with respect to the reflection arrangement of type $B_n$ is proved in~\cite[Lemma 6.3]{KVZ15}.
\end{proof}

We finish this section with the following observation discussed in the Introduction and related to Corollary~\ref{cor: number of vertices}. In the case $d=1$, the random subspace $\Ker A$ is a hyperplane a.s. It is easy to see that $N_n$, the number  of positive terms of the random walk $(S_i)_{i=1}^n$, is a.s. equal to the number of vertices of the simplex
\[
\{(\beta_1,\ldots,\beta_n) \in \R^n\colon 1 \geq  \beta_1 \geq  \beta_2 \geq  \ldots \geq  \beta_n \geq 0\}
\]
lying is the open half-space $\{ \beta \in \R^n: A \cdot \beta >0 \}$ with the boundary $\Ker A$.

\subsection{Proof of Theorem~\ref{theo:intersect_weyl_chambers}}\label{subsec:proof_thm_arcsine_geometric}
In the special case $n=k$ Theorem~\ref{theo:intersect_weyl_chambers} gives a formula for the number of Weyl chambers of type $B_n$ intersected non-trivially  by a linear subspace of codimension $d$ in general position. It was established in~\cite{KVZ15} using the theory of hyperplane arrangements. Let us state this result.

\begin{theorem}[\cite{KVZ15}]\label{theo:intersect_weyl_chambers_known}
Let $L\subset \R^n$ be a deterministic linear subspace of codimension $d$ in general position with respect to the reflection arrangement of type $B_n$. Let $Q$ be sampled randomly and uniformly among the $2^n n!$ closed Weyl chambers $C_{\eps,\sigma}^B$ of type $B_n$. Then,
\begin{align*}
\P [L\cap Q = \{0\}]
&\eqdef
\frac 1 {2^n n!}\sum_{\eps\in \{-1,+1\}^n} \sum_{\sigma\in \Sym(n)} \ind_{\{L\cap C_{\eps,\sigma}^B = \{0\}\}}\\
&=
\frac {2(B(n, d-1) + B(n, d-3) +\ldots)} {2^n n!}.
\end{align*}
\end{theorem}

The following combinatorial proof deduces the assertion of Theorem~\ref{theo:intersect_weyl_chambers} for general $k$ from Theorem~\ref{theo:intersect_weyl_chambers_known} without using any additional tools.

\vspace*{2mm}
\noindent 
\textbf{Enumeration of the $k$-faces of Weyl chambers of type $B_n$.}
Before proceeding to the proof, we introduce some notation and give few combinatorial examples.
Recall that the $k$-dimensional faces of the Weyl chamber $C_{\eps,\sigma}^B$ are denoted by $C_{\eps,\sigma}^B(i_1,\ldots,i_k)$; see~\eqref{eq:C_eps_sigma_face}.  It is important to stress that the $k$-face $C_{\eps,\sigma}^B(i_1,\ldots,i_k)$ may be a $k$-face of another Weyl chamber $C_{\eps',\sigma'}^B$ with some $(\eps',\sigma') \neq (\eps,\sigma)$.
\begin{example}\label{ex:k_face_chamber}
Consider the case $n=8$, $k=3$ and the convex cone given by the following set of conditions:
\begin{equation}\label{eq:example_face}
\underbrace{-\beta_2 =  \beta_4}_{\text{group } 1} \geq \underbrace{\beta_1 =  -\beta_6}_{\text{group } 2} \geq \underbrace{\beta_3}_{\text{group } 3} \geq \underbrace{\beta_5 = \beta_7 = \beta_8}_{\text{group } 4} = 0.
\end{equation}
This cone is a $3$-dimensional face of the Weyl chamber
\begin{equation}\label{eq:example_chamber}
-\beta_2 \geq  \beta_4 \geq \beta_1 \geq  -\beta_6 \geq \beta_3 \geq \beta_5 \geq \beta_7\geq \beta_8 \geq  0.
\end{equation}
However, it is also a $3$-face of
$$
\beta_4 \geq  -\beta_2 \geq -\beta_6 \geq  \beta_1 \geq \beta_3 \geq -\beta_8 \geq -\beta_5\geq \beta_7 \geq  0
$$
and, more generally,  any of the chambers obtained from~\eqref{eq:example_chamber} by permuting the $\beta$'s inside the groups $(-\beta_2,\beta_4)$, $(\beta_1,-\beta_6)$, $\beta_3$, $(\beta_5,\beta_7,\beta_8)$, and by changing any number of signs in the last group. The total number of such chambers is $2!2!1!3! 2^3$.
\end{example}
%The total number of such faces is $2^n n! \binom nk$, however not all of them are different.
%In the following we shall introduce
We now introduce an enumeration of all $k$-faces of all Weyl chambers such that each face is counted exactly once.  Let $\mathcal P_{n,k}$ be the set of all pairs $(I,\eta)$, where $I=(I_1,\ldots,I_{k+1})$ is a partition of the set $\{1,\ldots,n\}$ into $k+1$ disjoint distinguishable subsets (``groups'') such that $I_1,\ldots,I_k$ are non-empty, whereas $I_{k+1}$ may be empty or not, and $\eta: I_1\cup\ldots\cup I_k \to \{-1,+1\}$. We shall write $\eta_i := \eta(i)$. Given a pair $(I,\eta) \in \mathcal P_{n,k}$ define a closed $k$-dimensional convex cone
\begin{align*}
Q_{I,\eta}
:=
\{(\beta_1,\ldots,\beta_n) &\in\R^n \colon
\\
&
\text{for all } 1\leq l_1\leq l_2 \leq k \text{ and } i_1\in I_{l_1}, i_2\in I_{l_2}  \text{ we have } \eta_{i_1} \beta_{i_1} \geq \eta_{i_2} \beta_{i_2}\geq 0;\\
&\text{for all } i\in I_{k+1} \text{ we have } \beta_i = 0
\}.
\end{align*}
As a consequence of these conditions, for all $1\leq l\leq k$ and $i_1,i_2\in I_l$ we have $\eta_{i_1} \beta_{i_1} = \eta_{i_2} \beta_{i_2}$.

\begin{example}%[Example~\ref{ex:k_face_chamber} continued]
\label{ex:k_face_chamber_ctd}
If $n=8$, $k=3$  and the partition $I$ is given by $I_1= \{2,4\}$, $I_2 = \{1,6\}$, $I_3=\{3\}$, $I_4=\{5,7,8\}$, and the signs are $\eta_1 = \eta_3= \eta_4=+1$, $\eta_2=\eta_6=-1$, then the cone $Q_{I,\eta}$ is given by the set of inequalities~\eqref{eq:example_face}.
\end{example}

Given $(I,\eta) \in \mathcal P_{n,k}$ denote by $V_{I,\eta}$ the $k$-dimensional linear subspace of $\R^n$ spanned by $Q_{I,\eta}$, that is
\begin{align*}
V_{I,\eta}
:=
\{(\beta_1,\ldots,\beta_n) &\in\R^n \colon
\\
&
\text{for all } 1\leq l \leq k \text{ and } i_1, i_2\in I_{l}  \text{ we have } \eta_{i_1} \beta_{i_1} = \eta_{i_2} \beta_{i_2};\\
&\text{for all } i\in I_{k+1} \text{ we have } \beta_i = 0
\}.
\end{align*}
Using $\gamma_l := \eta_i \beta_i$, where $i\in I_l$ is arbitrary and $l=1,\ldots,k$, as coordinates on $V_{I,\eta}$ allows us to identify this linear space with $\R^k$. There is a natural decomposition of $V_{I,\eta}$ into $2^k k!$ Weyl chambers of type $B_k$ which have the form
$$
V_{I,\eta}(\zeta, \tau)
=
\{
(\beta_1,\ldots,\beta_n) \in V_{I,\eta}\colon
\zeta_1 \gamma_{\tau(1)} \geq \ldots\geq \zeta_k \gamma_{\tau(k)} \geq 0
\},
$$
where $\zeta\in \{-1,+1\}^k$,  $\tau \in \Sym(k)$. One of these chambers, corresponding to $\zeta_i=+1$, $\tau(i)=i$ for all $1\leq i \leq k$,  is $Q_{I,\eta}$.
\begin{example}\label{ex:k_face_chamber_ctd_ctd}
If the pair $(I,\eta)$ is the same as in Example~\ref{ex:k_face_chamber_ctd}, then the linear subspace $V_{I,\eta}$ is given by the following set of conditions
\begin{equation}\label{eq:example_V_I_eta}
\gamma_1:= \underbrace{-\beta_2 =  \beta_4}_{\text{group } 1}\in\R, \;\;
\gamma_2:= \underbrace{\beta_1 =  -\beta_6}_{\text{group } 2}\in\R, \;\;
\gamma_3:= \underbrace{\beta_3}_{\text{group } 3} \in\R,\;\;
\underbrace{\beta_5 = \beta_7 = \beta_8}_{\text{group } 4} = 0.
\end{equation}
It should be stressed that the linear subspaces $V_{I,\eta}$ are not pairwise different. For example, the set of conditions~\eqref{eq:example_V_I_eta} is clearly equivalent to the following one:
$$
\underbrace{-\beta_3}_{\text{former group } 3} \in\R,\;\;
\underbrace{\beta_1 =  -\beta_6}_{\text{former group } 2}\in\R, \;\;
\underbrace{\beta_2 =  -\beta_4}_{\text{former group } 1}\in\R, \;\;
\underbrace{\beta_5 = \beta_7 = \beta_8}_{\text{group } 4} = 0.
$$
More generally, we can interchange the first $k$ groups in an arbitrary way and multiply any number of groups by $\pm 1$, giving a total number of $2^k k!$ possibilities.
Clearly, the cone $Q_{I,\eta}$ given by~\eqref{eq:example_face} is one of the $2^k k!$ Weyl chambers which constitute $V_{I,\eta}$. However, as was explained above, there are other pairs $(I',\eta')$ such that $V_{I,\eta} = V_{I',\eta'}$ and  $Q_{I,\eta}$ coincides with one of the chambers $V_{I',\eta'}(\zeta,\tau)$.
\end{example}

\begin{proof}[\bf{Proof of Theorem~\ref{theo:intersect_weyl_chambers}}]
The cones $Q_{I,\eta}$, where $(I,\eta) \in \mathcal P_{n,k}$, are pairwise different and exhaust all $k$-dimensional faces of the Weyl chambers of type $B_n$.  The cone $Q_{I,\eta}$ belongs to $(\# I_1)!\ldots (\# I_{k+1})! 2^{\# I_{k+1}}$ Weyl chambers because, as was explained in Example~\ref{ex:k_face_chamber}, in the definition of the Weyl chamber containing $Q_{I,\eta}$ we can postulate any order of the elements $\eta_i\beta_i$, $i\in I_l$, for all $1\leq l\leq k$, and, additionally, we can postulate any order of the elements $\pm \beta_i$, $i\in I_{k+1}$, with arbitrary chosen signs. It follows that
\begin{equation}\label{eq:wspom1}
\sum_{\eps\in \{-1,+1\}^n} \sum_{\sigma\in \Sym(n)}
\sum_{F\in \mathcal F_k(C_{\eps,\sigma}^B)} \ind_{\{F\cap L = \{0\}\}}
=
\sum_{(I,\eta) \in \mathcal P_{n,k}} (\# I_1)!\ldots (\# I_{k+1})! 2^{\# I_{k+1}} \ind_{\{Q_{I,\eta}\cap L= \{0\}\}}.
\end{equation}

In the rest of the proof we compute the right-hand side of~\eqref{eq:wspom1}. We may suppose that $k\geq d$ because otherwise Theorem~\ref{theo:intersect_weyl_chambers} becomes trivial ($L$ intersects all $k$-faces trivially). Then, the codimension of $L\cap V_{I,\eta}$ in $V_{I,\eta}$ (which is an element of $\mathcal L(B_n)$) is $d$ because $L$ is in general position with respect to the reflection arrangement of type $B_n$ in $\R^n$; see~\eqref{eq:def_gen_pos}.
Also,  $L\cap V_{I,\eta}$  is in general position with respect to the reflection arrangement of type $B_k$ in $V_{I,\eta}$, as can be checked using the definition.
It follows from Theorem~\ref{theo:intersect_weyl_chambers_known} applied to the linear subspace $L\cap V_{I,\eta}\subset V_{I,\eta}$  that
$$
\sum_{\zeta\in \{-1,+1\}^k} \sum_{\tau \in \Sym(k)} \ind_{\{L \cap V_{I,\eta} (\zeta, \tau) = \{0\}\}}
= 2(B(k,d-1) + B(k,d-3) + \ldots).
$$
Multiplying this equality by $(\# I_1)!\ldots (\# I_{k+1})! 2^{\# I_{k+1}}$ and taking the sum over all $(I,\eta)\in \mathcal P_{n,k}$ we obtain
\begin{multline}\label{eq:wspom_LHS_RHS}
\sum_{(I,\eta)\in\mathcal P_{n,k}} \sum_{\zeta\in \{-1,+1\}^k} \sum_{\tau \in \Sym(k)}  (\# I_1)!\ldots (\# I_{k+1})! 2^{\# I_{k+1}} \ind_{\{L \cap V_{I,\eta} (\zeta, \tau) = \{0\}\}}
\\=
2(B(k,d-1) + B(k,d-3) + \ldots)
\sum_{(I,\eta)\in\mathcal P_{n,k}} (\# I_1)!\ldots (\# I_{k+1})! 2^{\# I_{k+1}}.
\end{multline}
Let us look at the triple sum on the left-hand side of~\eqref{eq:wspom_LHS_RHS}. Since any convex cone $Q_{I,\eta}$ can be represented in the form $V_{I',\eta'}(\zeta', \tau')$ in $2^k k!$ different ways,  we have
\begin{equation}\label{eq:wspom3}
\text{LHS\eqref{eq:wspom_LHS_RHS}} =
2^k k! \sum_{(I,\eta) \in \mathcal P_{n,k}} (\# I_1)!\ldots (\# I_{k+1})! 2^{\# I_{k+1}} \ind_{\{Q_{I,\eta}\cap L= \{0\}\}}.
\end{equation}
Let us now compute the sum on the right-hand side of~\eqref{eq:wspom_LHS_RHS}.
For $j_1,\ldots,j_k\in\N$ such that $j_1+\ldots+j_k\leq n$ denote by $\mathcal P_{n,k}(j_1,\ldots,j_k)$ the set of all pairs $(I,\eta) \in \mathcal P_{n,k}$ such that $\#I_1=j_1, \ldots, \#I_k = j_k$. Let $j_{k+1} = n-j_1-\ldots-j_k$. The number of elements in $\mathcal P_{n,k}(j_1,\ldots,j_k)$ is given by
$$
\# \mathcal P_{n,k} (j_1,\ldots,j_k) = \frac {n!2^{j_1+\ldots+j_k}}{j_1!\ldots j_{k+1}!} .
$$
It follows that
\begin{align}
\sum_{(I,\eta)\in\mathcal P_{n,k}} (\# I_1)!\ldots (\# I_{k+1})! 2^{\# I_{k+1}}
&=
\sum_{\substack{j_1,\ldots,j_k\in\N\\j_1+\ldots+j_k \leq n}}
\sum_{(I,\eta) \in\mathcal P_{n,k}(j_1,\ldots,j_k)} (\# I_1)!\ldots (\# I_{k+1})! 2^{\# I_{k+1}} \notag\\
&=
\sum_{\substack{j_1,\ldots,j_k\in\N\\j_1+\ldots+j_k \leq n}} \frac {n!2^{j_1+\ldots+j_k} }{j_1!\ldots j_{k+1}!} \cdot  j_1!\ldots j_{k+1}! 2^{n-(j_1+\ldots+j_k)}\notag\\
&= \sum_{\substack{j_1,\ldots,j_k\in\N\\j_1+\ldots+j_k \leq n}} n! 2^n\notag\\
&= n! 2^n \binom nk. \label{eq:wspom2}
\end{align}
Taking~\eqref{eq:wspom_LHS_RHS}, \eqref{eq:wspom3}, \eqref{eq:wspom2}  together yields
$$
\sum_{(I,\eta) \in \mathcal P_{n,k}} (\# I_1)!\ldots (\# I_{k+1})! 2^{\# I_{k+1}} \ind_{\{Q_{I,\eta}\cap L= \{0\}\}}
=
2^n n!\binom nk \frac{2(B(k,d-1) + B(k,d-3) + \ldots)}{2^k k!}.
$$
In view of~\eqref{eq:wspom1}, this completes the proof of Theorem~\ref{theo:intersect_weyl_chambers}.
\end{proof}

\subsection{The reflection group and Weyl chambers of type $A_{n-1}$}
We proceed to a geometric interpretation of Theorem~\ref{theo:arcsine_multidim_Br}. In the same way as random walks are related to the reflection group of type $B_n$, random bridges are related to the reflection group of type $A_{n-1}$. We start by recalling some definitions.

The \emph{reflection group $\mathcal G(A_{n-1})$} is the symmetric group $\Sym(n)$ acting on $\R^n$ by permuting the coordinates in an arbitrary way. The number of elements in $\mathcal G(A_{n-1})$ is $n!$. Note that this group leaves the hyperplane
$$
L_0 := \{(\beta_1,\ldots,\beta_n) \in \R^n\colon \beta_1 +  \ldots + \beta_n = 0\}
$$
invariant.
The \emph{fundamental Weyl chamber of type $A_{n-1}$} is the convex cone
$$
\mathcal C(A_{n-1}) := \{(\beta_1,\ldots,\beta_n)\in\R^n \colon \beta_{1} <  \ldots < \beta_{n}\}.
$$
Acting on the closure of this cone by the elements of the group $\mathcal G(A_{n-1})$, we obtain the closed Weyl chambers of type $A_{n-1}$  given by
\begin{equation}\label{eq:C_sigma_def}
C_{\sigma}^A := \{(\beta_1,\ldots,\beta_n) \in\R^n \colon \beta_{\sigma(1)} \geq  \ldots \geq \beta_{\sigma(n)}\},
\end{equation}
where $\sigma\in \Sym(n)$. Note that  $\cup_{\sigma\in \Sym(n)}C_\sigma^A = \R^n$, and the interiors of these convex cones are disjoint.
The \emph{reflection arrangement $\mathcal A(A_{n-1})$ of type $A_{n-1}$} consists of the hyperplanes
\begin{equation}\label{eq:reflection_arrangement_A}
\{\beta_i =\beta_j\}, \;\; 1\leq i < j \leq n.
\end{equation}
The lattice $\mathcal L(A_{n-1})$ generated by the reflection arrangement $\mathcal A(A_{n-1})$ and the notion of general position are defined in the same way as in the $B_n$-case.
The next result is an analogue of Theorem~\ref{theo:intersect_weyl_chambers} for Weyl chambers of type $A_{n-1}$.
\begin{theorem}\label{theo:intersect_weyl_chambers_Br}
Let $L\subset \R^n$ be a deterministic linear subspace of codimension $d$ in general position with respect to the reflection arrangement~\eqref{eq:reflection_arrangement_A} of type $A_{n-1}$. Let $Q$ be sampled randomly and uniformly among the $n!$ closed Weyl chambers $C_{\sigma}^A$ of type $A_{n-1}$. Then, the expected number of $k$-dimensional faces of $Q$ intersected  by $L$ in a trivial way is given by
\begin{align*}
\E \left[\sum_{F\in \mathcal F_k(Q)} \ind_{\{F\cap L = \{0\}\}}\right]
&\eqdef
\frac{1}{n!} \sum_{\sigma\in \Sym(n)}
\sum_{F\in \mathcal F_k(C_{\sigma}^A)} \ind_{\{F\cap L = \{0\}\}}\\
&=
\frac 2 {k!} \binom  {n-1} {k-1}  \left(\stirling{k}{d-1} + \stirling{k}{d-3} +\ldots\right),
\end{align*}
where the $\stirling{k}{j}$'s are the Stirling numbers of the first kind defined in Theorem~\ref{theo:arcsine_multidim_Br}.
\end{theorem}
In the special case $k=n$, Theorem~\ref{theo:intersect_weyl_chambers_Br} reduces to the formula, which was proved in~\cite{KVZ15}:
\begin{align}
\P [L\cap Q = \{0\}]
\eqdef
\frac 1 {n!} \sum_{\sigma\in \Sym(n)} \ind_{\{L\cap C_{\sigma}^A = \{0\}\}}
=
\frac 2 {n!} \left(\stirling{n}{d-1} + \stirling{n}{d-3} +\ldots\right)
.
\label{eq:L_cap_Q_Br}
\end{align}

\subsection{Proof of Theorem~\ref{theo:arcsine_multidim_Br} given Theorem~\ref{theo:intersect_weyl_chambers_Br}}\label{017}
We shall need a short notation for one of the closed Weyl chambers of type $A_{n-1}$:
$$
C^A := \{(\beta_1,\ldots,\beta_n) \in \R^n\colon \beta_1  \geq \ldots \geq \beta_n\}.
$$
The next lemma is an analogue of Lemma~\ref{lem:geometric_interpretation}. The case $k=n$ of this lemma appeared in~\cite{KVZ15}.

\begin{lemma}\label{lem:geometric_interpretation_Br}
Let $x_1,\ldots,x_n\in\R^d$, where $n\geq 2$, be arbitrary vectors such that $x_1+\ldots+x_n = 0$. Denote by $s_i = x_1+\ldots+x_i$, $1\leq i\leq n$, their partial sums. Let $A: \R^n \to \R^d$ be a linear operator defined on the standard basis $e_1,\ldots,e_n$ of $\R^n$ by $A e_1 =x_1, \ldots, A e_n = x_n$.
Then the number of collections $1\leq i_1<\ldots < i_k\leq n-1$ such that $0\in \conv(s_{i_1},\ldots,s_{i_k})$ is equal to the number of $(k+1)$-dimensional faces $F$ of the convex cone $C^A$ intersected non-trivially by the linear subspace $L_0\cap \Ker A$.
% (meaning that $F \cap L_0\cap \Ker A \neq \{0\}$).
\end{lemma}

\begin{proof}
For a given collection of indices $1\leq i_1<\ldots < i_k\leq n-1$ we have $0\in \conv(s_{i_1},\ldots,s_{i_k})$ if and only if there exist $\alpha_1,\ldots,\alpha_k\geq 0$ (not all of them being $0$) such that $\alpha_1 s_{i_1} + \ldots+ \alpha_k s_{i_k} = 0$, or, equivalently,
$$
\alpha_1 (x_1+\ldots+x_{i_1}) + \alpha_2 (x_1+\ldots+x_{i_2}) + \ldots + \alpha_k (x_1+\ldots+x_{i_k}) = 0.
$$
After rearranging the terms, we can rewrite this condition as $\beta_1 x_1 + \ldots + \beta _n x_n = 0$,
where
\begin{align*}
&\beta_1 =\ldots = \beta_{i_1} := \alpha_1+\ldots+\alpha_k - b,\\
&\beta_{i_1+1} =\ldots = \beta_{i_2} := \alpha_2+\ldots+\alpha_k - b,\\
&\ldots,\\
&\beta_{i_{k-1}+1} =\ldots = \beta_{i_k} := \alpha_k -b,\\
&\beta_{i_k + 1} = \ldots = \beta_n :=-b,
\end{align*}
and $b\in\R$ can be arbitrary due to the assumption $x_1+\ldots+x_n = 0$. Choose  $b := \frac 1n (i_1\alpha_1 + \ldots + i_k \alpha_k)$, which ensures that $\beta_1+\ldots+\beta_n = 0$.
Our conditions on the $\alpha_i$'s translate into the following equivalent conditions on the $\beta_i$'s:
\begin{align}
&\beta_1 = \ldots = \beta_{i_1} \geq \beta_{i_1+1} =\ldots = \beta_{i_2} \geq \ldots \geq \beta_{i_{k-1}+1} = \ldots =\beta_{i_k} \geq \beta_{i_k+1} =\ldots = \beta_n, \label{eq:face_of_Weyl_chamber_Br1}\\
& \beta_1+\ldots+\beta_n = 0, \label{eq:face_of_Weyl_chamber_Br2}
\end{align}
where at least one inequality in~\eqref{eq:face_of_Weyl_chamber_Br1} should be strict, i.e.\ $(\beta_1,\ldots,\beta_n)\neq 0$.  That is, we have $0\in \conv(s_{i_1},\ldots,s_{i_k})$ if and only if $F\cap L_0 \cap \Ker A \neq \{0\}$, where $F\subset \R^n$ is the closed convex cone defined by~\eqref{eq:face_of_Weyl_chamber_Br1}.  Since any such $F$ is a $(k+1)$-dimensional face of the Weyl chamber $C^A$ and, conversely, any $(k+1)$-face has this form, we obtain the required statement.
\end{proof}

\begin{proof}[Proof of Theorem~\ref{theo:arcsine_multidim_Br} given Theorem~\ref{theo:intersect_weyl_chambers_Br}]
Let $A: \R^n \to \R^d$ be a random linear operator defined on the standard basis $e_1,\ldots,e_n$ of $\R^n$ by $A e_1 =\xi_1, \ldots, A e_n = \xi_n$.  By Lemma~\ref{lem:geometric_interpretation_Br},
$$
M_{n,k}^{(d)} =
\frac{1}{2} \sum_{1\leq i_1 < \ldots < i_k < n} \ind_{\{0\notin \conv(S_{i_1},\ldots, S_{i_k})\}}
=
\frac 12 \sum_{F\in \mathcal F_{k+1}(C^A)} \ind_{\{F \cap L_0 \cap \Ker A = \{0\}\}}
\quad   \text{a.s.},
$$
where we also used that $\# \mathcal F_{k+1}(C^A) = \binom{n-1}{k}$.
By the exchangeability assumption~$(\text{Ex})$, for every closed Weyl chamber $C_{\sigma}^A$,
$$
M_{n,k}^{(d)} \eqdistr \frac 12 \sum_{F\in \mathcal F_{k+1}(C_{\sigma}^A)} \ind_{\{F\cap L_0 \cap \Ker A = \{0\}\}}.
$$
Taking the expectation and then the mean over all $n!$ permutations $\sigma\in \Sym(n)$, we obtain
\begin{align*}
2\E M_{n,k}^{(d)}
&=
\E \left[\frac{1}{n!} \sum_{\sigma\in \Sym(n)}
\sum_{F\in \mathcal F_{k+1}(C_{\sigma}^A)} \ind_{\{F\cap L_0\cap \Ker A = \{0\}\}}\right] \\
&=
\frac 2 {(k+1)!} \binom  {n-1} {k}  \left(\stirling{k+1}{d} + \stirling{k+1}{d-2} +\ldots\right),
\end{align*}
where the second equality is by Theorem~\ref{theo:intersect_weyl_chambers_Br} applied to the linear subspace $L:= L_0\cap \Ker A$.
The fact that with probability one, the linear subspace $L_0\cap \Ker A$ is in general position with respect to the hyperplane arrangement $\mathcal A(A_{n-1})$ and the codimension of this subspace in $\R^n$ is $d+1$ is proved in~\cite[Lemma 6.2]{KVZ15}.
\end{proof}

\subsection{Proof of Theorem~\ref{theo:intersect_weyl_chambers_Br}}\label{subsec:proof_thm_arcsine_geometric_Br}
The proof is similar to that of Theorem~\ref{theo:intersect_weyl_chambers}, but several simplifications are possible.

\vspace*{2mm}
\noindent
\textbf{Enumeration of the $k$-faces of Weyl chambers of type $B_n$.} 
Again, we start with notation and examples. The $k$-dimensional faces of the Weyl chamber $C_{\sigma}^A$, see~\eqref{eq:C_sigma_def},  are enumerated by collections $1\leq i_1 < \ldots < i_{k-1} \leq n-1$  as follows:
\begin{multline}\label{eq:C_eps_sigma_face_Br}
C_{\sigma}^A(i_1,\ldots,i_{k-1}) :=
\{(\beta_1,\ldots,\beta_n)\in\R^n \colon
\beta_{\sigma(1)} = \ldots = \beta_{\sigma(i_1)}
\\\geq
\beta_{\sigma(i_1+1)} =\ldots = \beta_{\sigma(i_2)}
\geq
\ldots
%\geq
%\beta_{\sigma(i_{k-1}+1)} = \ldots = \beta_{\sigma(i_k)}
\geq
\beta_{\sigma(i_{k-1}+1)} =\ldots = \beta_{\sigma(n)}\}.
\end{multline}
The next example shows  that the $k$-face $C_{\sigma}^A(i_1,\ldots,i_{k-1})$ may be a $k$-face of another Weyl chamber $C_{\sigma'}^A$ with some $\sigma' \neq \sigma$.
\begin{example}\label{ex:k_face_chamber_Br}
Consider the case $n=5$, $k=3$ and the convex cone given by the following set of conditions:
\begin{equation}\label{eq:example_face_Br}
\underbrace{\beta_2 =  \beta_4}_{\text{group } 1} \geq \underbrace{\beta_1 =  \beta_5}_{\text{group } 2} \geq \underbrace{\beta_3}_{\text{group } 3}.
% \geq \underbrace{\beta_5 = \beta_7 = \beta_8}_{\text{group } 4}.
\end{equation}
This cone is a $3$-dimensional face of the Weyl chamber
\begin{equation}\label{eq:example_chamber_Br}
\beta_2 \geq  \beta_4 \geq \beta_1 \geq  \beta_5 \geq \beta_3.
% \geq \beta_5 \geq \beta_7\geq \beta_8.
\end{equation}
However, it is also a $3$-face of
$$
\beta_4 \geq  \beta_2 \geq \beta_5 \geq  \beta_1 \geq \beta_3.% \geq \beta_8 \geq \beta_5\geq \beta_7
$$
and, more generally,  any of the chambers obtained from~\eqref{eq:example_chamber_Br} by permuting the $\beta_i$'s inside the groups $(\beta_2,\beta_4)$, $(\beta_1,\beta_5)$, $\beta_3$. The total number of such chambers is $2!2!1!$.
\end{example}
%The total number of such faces is $2^n n! \binom nk$, however not all of them are different.
Next we shall introduce a notation for all $k$-faces of all Weyl chambers such that each face is counted exactly once.  Let $\mathcal R_{n,k}$ be the set of all partitions $I=(I_1,\ldots,I_{k})$ of the set $\{1,\ldots,n\}$ into $k$ disjoint non-empty distinguishable subsets $I_1,\ldots,I_k$.   Given a partition $I \in \mathcal R_{n,k}$, define the closed $k$-dimensional convex cone
\begin{align*}
Q_{I}
:=
\{(\beta_1,\ldots,\beta_n) \in\R^n \colon
\text{for all } 1\leq l_1\leq l_2 \leq k \text{ and } i_1\in I_{l_1}, i_2\in I_{l_2}  \text{ we have } \beta_{i_1} \geq \beta_{i_2}
\}.
\end{align*}
It follows from these conditions that for all $1\leq l\leq k$ and $i_1,i_2\in I_l$ we have $\beta_{i_1} =  \beta_{i_2}$.

\begin{example}%[Example~\ref{ex:k_face_chamber} continued]
\label{ex:k_face_chamber_ctd_Br}
If $n=5$, $k=3$  and the partition $I$ is given by $I_1= \{2,4\}$, $I_2 = \{1,5\}$, $I_3=\{3\}$, then the cone $Q_{I}$ is given by the set of inequalities~\eqref{eq:example_face_Br}.
\end{example}

For a partition $I \in \mathcal R_{n,k}$, denote by $W_{I}$ the $k$-dimensional linear subspace of $\R^n$ spanned by $Q_i$, that is
\begin{align*}
W_{I}
:=
\{(\beta_1,\ldots,\beta_n) \in\R^n \colon
\text{for all } 1\leq l \leq k \text{ and } i_1, i_2\in I_{l}  \text{ we have } \beta_{i_1} = \beta_{i_2}
\}.
\end{align*}
Using $\gamma_l := \beta_i$, where $i\in I_l$ is arbitrary and $l=1,\ldots,k$, as coordinates on $W_{I}$ allows us to identify this linear space with $\R^k$. There is a natural decomposition of $W_{I}$ into $k!$ Weyl chambers of type $A_{k-1}$ of the form
$$
W_{I}(\tau)
=
\{
(\beta_1,\ldots,\beta_n) \in W_{I}\colon
\gamma_{\tau(1)} \geq \ldots \geq \gamma_{\tau(k)}
\},
$$
where  $\tau \in \Sym(k)$. One of these chambers, corresponding to the identity permutation $\tau(i)=i$ for all $1\leq i \leq k$,  is $Q_{I}$.
\begin{example}\label{ex:k_face_chamber_ctd_ctd_Br}
If the partition $I$ is the same as in Example~\ref{ex:k_face_chamber_ctd_Br}, then the linear space $W_{I}$ is given by the following set of conditions
\begin{equation}\label{eq:example_V_I_eta_Br}
\gamma_1:= \underbrace{\beta_2 =  \beta_4}_{\text{group } 1}\in\R, \;\;
\gamma_2:= \underbrace{\beta_1 =  \beta_5}_{\text{group } 2}\in\R, \;\;
\gamma_3:= \underbrace{\beta_3}_{\text{group } 3} \in\R.
\end{equation}
The spaces $W_{I}$ are not pairwise distinct. For example, the set of conditions~\eqref{eq:example_V_I_eta_Br} is  equivalent to the following one:
$$
\underbrace{\beta_3}_{\text{former group } 3} \in\R,\;\;
\underbrace{\beta_1 =  \beta_5}_{\text{former group } 2}\in\R, \;\;
\underbrace{\beta_2 =  \beta_4}_{\text{former group } 1}\in\R.
$$
More generally, we can interchange the $k$ groups of equal $\beta_i$'s in an arbitrary way, so that in the list $W_I$, $I\in \mathcal R_{n,k}$, each linear subspace appears $k!$ times under different names.
%Clearly, the cone $Q_{I}$ given by~\eqref{eq:example_face_Br} is one of the $k!$ Weyl chambers which constitute $W_{I}$. However, there exist other partitions $I'$ such that $Q_{I}$ coincides with one of the chambers $W_{I'}(\tau)$.
\end{example}

\begin{proof}[\bf{Proof of Theorem~\ref{theo:intersect_weyl_chambers_Br}}]
The cones $Q_{I}$, where $I\in \mathcal R_{n,k}$, are pairwise distinct and exhaust all $k$-dimensional faces of the Weyl chambers of type $A_{n-1}$.  The cone $Q_{I}$ belongs to $(\# I_1)!\ldots (\# I_{k})!$ different Weyl chambers
because in the definition of the Weyl chamber containing $Q_{I}$ we can impose an arbitrary ordering of the elements $\beta_i$, $i\in I_l$, for all $1\leq l\leq k$; see Example~\ref{ex:k_face_chamber_Br}.
It follows that
\begin{equation}\label{eq:wspom1_Br}
\sum_{\sigma\in \Sym(n)}
\sum_{F\in \mathcal F_k(C_{\sigma}^A)} \ind_{\{F\cap L = \{0\}\}}
=
\sum_{I \in \mathcal R_{n,k}} (\# I_1)!\ldots (\# I_{k})! \ind_{\{Q_{I}\cap L= \{0\}\}}.
\end{equation}

In the rest of the proof we compute the right-hand side of~\eqref{eq:wspom1_Br}. We may assume that $k\geq d+1$ since otherwise Theorem~\ref{theo:intersect_weyl_chambers_Br} is trivial.  Recall that the linear space $L$ has codimension $d$ in $\R^n$ and is in general position with respect to the reflection arrangement of type $A_{n-1}$ in $\R^n$. It follows from the definition of the general position, see~\eqref{eq:def_gen_pos}, that the linear subspace $L\cap W_{I}\subset W_{I}$ has codimension $d$ in $W_I$ and is in general position with respect to the reflection arrangement of type $A_{k-1}$ in $W_I$.
It follows from~\eqref{eq:L_cap_Q_Br}  applied to $L\cap W_{I}\subset W_{I}$  that
$$
\sum_{\tau \in \Sym(k)} \ind_{\{L \cap W_{I} (\tau) = \{0\}\}}
=
2 \left(\stirling{k}{d-1} + \stirling{k}{d-3} +\ldots\right).
$$
Multiplying this equality by $(\# I_1)!\ldots (\# I_{k})!$ and taking the sum over all partitions $I\in \mathcal R_{n,k}$, we obtain
\begin{multline}\label{eq:wspom_LHS_RHS_Br}
\sum_{I\in\mathcal R_{n,k}} \sum_{\tau \in \Sym(k)}  (\# I_1)!\ldots (\# I_{k})!  \ind_{\{L \cap W_{I} (\tau) = \{0\}\}}
\\=
2 \left(\stirling{k}{d-1} + \stirling{k}{d-3} +\ldots\right)
\sum_{I\in\mathcal R_{n,k}} (\# I_1)!\ldots (\# I_{k})!.
\end{multline}

Since any $k$-face $Q_I$ can be represented as $W_{I'}(\tau')$ in $k!$ ways, see Example~\ref{ex:k_face_chamber_ctd_ctd_Br}, and the sets  $I_1,\ldots, I_k$ are (up to their order) the same in all representations,
\begin{equation}\label{eq:wspom3_Br}
\text{LHS\eqref{eq:wspom_LHS_RHS_Br}} =
k! \sum_{I \in \mathcal R_{n,k}} (\# I_1)!\ldots (\# I_{k})! \ind_{\{Q_{I}\cap L= \{0\}\}}.
\end{equation}
Let us now compute the sum on the right-hand side of~\eqref{eq:wspom_LHS_RHS_Br}.
For $j_1,\ldots,j_k\in\N$ such that $j_1+\ldots+j_k = n$ denote by $\mathcal R_{n,k}(j_1,\ldots,j_k)$ the set of all partitions $I \in \mathcal R_{n,k}$ such that $\#I_1=j_1, \ldots, \#I_k = j_k$.  The number of elements in $\mathcal R_{n,k}(j_1,\ldots,j_k)$ is given by
$$
\# \mathcal R_{n,k} (j_1,\ldots,j_k) = \frac {n!}{j_1!\ldots j_{k}!} .
$$
It follows that
\begin{align}
\sum_{I\in\mathcal R_{n,k}} (\# I_1)!\ldots (\# I_{k})!
&=
\sum_{\substack{j_1,\ldots,j_k\in\N\\j_1+\ldots+j_k = n}}
\sum_{I \in\mathcal R_{n,k}(j_1,\ldots,j_k)} (\# I_1)!\ldots (\# I_{k})!  \notag\\
&=
\sum_{\substack{j_1,\ldots,j_k\in\N\\j_1+\ldots+j_k = n}} \frac {n!}{j_1!\ldots j_{k}!} \cdot  j_1!\ldots j_{k}!\notag\\
%&= \sum_{\substack{j_1,\ldots,j_k\in\N\\j_1+\ldots+j_k = n}} n! \notag\\
&= n! \binom {n-1}{k-1}. \label{eq:wspom2_Br}
\end{align}
Taking~\eqref{eq:wspom_LHS_RHS_Br}, \eqref{eq:wspom3_Br}, \eqref{eq:wspom2_Br}  together yields
$$
\sum_{I\in \mathcal R_{n,k}} (\# I_1)!\ldots (\# I_{k})!  \ind_{\{Q_{I}\cap L= \{0\}\}}
=
n!\binom {n-1}{k-1} \frac 2 {k!} \left(\stirling{k}{d-1} + \stirling{k}{d-3} +\ldots\right).
$$
In view of~\eqref{eq:wspom1_Br}, this completes the proof of Theorem~\ref{theo:intersect_weyl_chambers_Br}.
\end{proof}

\section{Proofs by reduction to non-absorption probability}\label{sec:proofs}
In this section we give alternative proofs of Theorems~\ref{theo:arcsine_multidim} and~\ref{theo:arcsine_multidim_Br}.

\subsection{Proof of Theorem~\ref{theo:arcsine_multidim}}\label{sec:proofs_B_n}
This proof rests on the following result obtained in~\cite{KVZ15}, which is the special case of Theorem~\ref{theo:arcsine_multidim} with $k=n$:
\begin{theorem}[\cite{KVZ15}]\label{theo:absorption}
Let $(S_i)_{i=1}^n$ be a random walk in $\R^d$ satisfying assumptions $(\pm\text{Ex})$ and $(\text{GP})$. Then
$$
\P[0\notin \conv (S_1,\ldots,S_n)] = \frac {2(B(n,d-1) + B(n, d-3)+\ldots)}{2^n n!}.
$$
\end{theorem}

We want to use this result to compute $\P[0\notin \conv(S_{i_1},\ldots,S_{i_k})]$ but we cannot apply it directly because the increments $\xi_1':=S_{i_1}, \dots, \xi_k':=S_{i_k} - S_{i_{k-1}}$ in general are not exchangeable. We restore the exchangeability by introducing an additional random reshuffling of the $\xi_i'$'s, which is the main idea of the following proof.
\begin{proof}[Proof of Theorem~\ref{theo:arcsine_multidim}]
Take some $1\leq i_1 < \ldots < i_k \leq n$ and let $i_0:=0$.
We subdivide the collection $(\xi_1,\ldots,\xi_n)$  into $k$ groups of lengths
$$
j_1:=i_1>0, \;\; j_2:=i_2-i_1>0,\;\; \ldots, \;\; j_k:=i_k-i_{k-1}>0
$$
and one group of length  $n-(j_1+\ldots+j_k)= n-i_k \geq 0$ as follows:
\begin{equation}\label{eq:groups_1}
\underbrace{(\xi_1,\ldots,\xi_{i_1})}_{\text{length $j_1$}},
\;\;
\underbrace{(\xi_{i_1+1},\ldots,\xi_{i_2})}_{\text{length $j_2$}},
\;\; \ldots,\;\;
\underbrace{(\xi_{i_{k-1} +1},\ldots, \xi_{i_k})}_{\text{length $j_k$}},
\;\;
 (\xi_{i_{k}+1},\ldots,\xi_n).
\end{equation}
The last, $(k+1)$-st group, will be mostly ignored in the sequel since the outcome of the corresponding terms in the definition of $M_{n,k}^{(d)}$ does not depend on these values.

Denote by $\xi_l'$ the sum of the $\xi_i$'s in the $l$-th group above, that is
$$
\xi_l' := \xi_{i_{l-1} +1} +\ldots + \xi_{i_l}, \quad l=1,\ldots, k.
$$
Let $\tau$ be a random permutation that is uniformly distributed on $\Sym(k)$ and independent of $(\xi_1,\ldots,\xi_n)$. We reshuffle the $\xi_i'$'s as follows:
$$
\eta_i:=\xi_{\tau(i)}', \quad i=1,\ldots, k.
$$
We claim that these $k$ random vectors satisfy the symmetric exchangeability assumption $(\pm\text{Ex})$ (with $k$ substituted for $n$). The exchangeability is by the construction and the symmetry follows by the symmetry of the joint distribution of $(\xi_1, \ldots, \xi_n)$. Let us give the details.

Let $\sigma_1$ be any permutation of the set $\{1,\ldots,k\}$ and let $\eps=(\eps_1,\ldots,\eps_k) \in \{-1,+1\}^k$ be a sequence of $\pm1$'s of length $k$. By the fact that $\tau \sigma_1 \eqdistr \tau$, we have the distributional identity
\begin{equation*}
(\eps_1 \eta_{\sigma_1(1)}, \ldots, \eps_k \eta_{\sigma_1(k)} ) = (\eps_1 \xi_{\tau(\sigma_1(1))}', \ldots, \eps_k \xi_{\tau(\sigma_1(k))}')
\eqdistr (\eps_1 \eta_1, \ldots, \eps_k \eta_k ),
\end{equation*}
from which we deduce
\begin{equation}\label{eq:eta_symm_exchang}
\P[(\eps_1 \eta_{\sigma_1(1)}, \ldots, \eps_k \eta_{\sigma_1(k)} ) \in \cdot] = \frac{1}{k!} \sum_{\sigma \in \Sym(k)} \P[(\eps_1 \xi_{\sigma(1)}', \ldots, \eps_k \xi_{\sigma(k)}' ) \in \cdot \, ].
\end{equation}
%where the notation in the last equation stresses that the average of distributions rather than vectors is taken.

Let $\sigma \in \Sym(k)$ be any permutation of length $k$. Permute the first $k$ groups of the list \eqref{eq:groups_1} according to $\sigma$ and then multiply the reordered groups by $\eps_1,\ldots,\eps_k$, respectively. The $(k+1)$-st group stays unchanged. Clearly, our symmetric exchangeability assumption $(\pm\text{Ex})$ on $(\xi_1,\ldots,\xi_n)$ implies that this operation does not change the distribution. That is, ignoring the last group, we have the  distributional equality
$$
(\xi_1,\ldots,\xi_{i_k})
\eqdistr
(\underbrace{\eps_1 \xi_{i_{\sigma(1)-1} + 1},\ldots, \eps_1 \xi_{i_{\sigma(1)}}}_{\text{length $j_{\sigma(1)}$}},
\underbrace{\eps_2 \xi_{i_{\sigma(2)-1} + 1},\ldots, \eps_2 \xi_{i_{\sigma(2)}}}_{\text{length $j_{\sigma(2)}$}},
\ldots,
\underbrace{\eps_k \xi_{i_{\sigma(k)-1} + 1},\ldots, \eps_k \xi_{i_{\sigma(k)}}}_{\text{length $j_{\sigma(k)}$}}).
$$
Since the sum of random vectors in the $l$-th group is $\eps_l \xi_{\sigma(l)}'$, we see that the distributions under the sum in \eqref{eq:eta_symm_exchang} do not depend on $\eps_1, \ldots, \eps_k$. This proves the stated symmetric exchangeability of $(\eta_1, \ldots, \eta_k)$. This also gives the distributional identity
\begin{equation}
\label{eq: distr =}
(S_{j_{\sigma(1)}},S_{j_{\sigma(1)}+j_{\sigma(2)}},\ldots,S_{j_{\sigma(1)}+\ldots+j_{\sigma(k)}}) \eqdistr (\xi_{\sigma(1)}', \xi_{\sigma(1)}' + \xi_{\sigma(2)}', \ldots, \xi_{\sigma(1)}' + \ldots + \xi_{\sigma(k)}' ).
\end{equation}

Now introduce the new random walk
$$T_i:= \eta_1 + \ldots +\eta_i,  \quad i=1,\ldots, k.$$
From \eqref{eq: distr =} it follows that the random walk $(T_i)_{i=1}^k$ satisfies the general position assumption $(\text{GP})$ (with $k$ substituted for $n$) since $(S_i)_{i=1}^n$ satisfies $(\text{GP})$. Then Theorem~\ref{theo:absorption} applies to $(T_i)_{i=1}^k$, and by taking the mean over all $\sigma \in \Sym(k)$ in \eqref{eq: distr =}, we obtain
\begin{multline}
\label{eq: symmetrized}
\frac 1 {k!} \sum_{\sigma \in \Sym(k)}
\P[0\notin \conv (S_{j_{\sigma(1)}},S_{j_{\sigma(1)}+j_{\sigma(2)}},\ldots,S_{j_{\sigma(1)}+\ldots+j_{\sigma(k)}})]
\\
\begin{aligned}
&= \P[0\notin \conv (T_1, \ldots, T_k)] \\
&=\frac {2(B(k, d-1) + B(k, d-3) +\ldots)} {2^k k!}.
\end{aligned}
\end{multline}

Finally, sum the equations above over all tuples $(j_1,\ldots,j_k)$ from the following set:
$$
J_{n,k} := \{(j_1,\ldots,j_k)\in \N^k\colon j_1+\ldots+j_k\leq n\}.
$$
Since the cardinality of $J_{n,k}$ is $\binom nk$ and the last expression in \eqref{eq: symmetrized} does not depend on $(j_1,\ldots,j_k)$, we obtain
\begin{multline*}
\frac {2} {2^k k!}\binom nk (B(k, d-1) + B(k, d-3) +\ldots)
\\
\begin{aligned}
&=
\frac 1 {k!} \sum_{\sigma \in \Sym(k)}  \sum_{(j_1,\ldots,j_k) \in J_{n,k}}
\P[0\notin \conv (S_{j_{\sigma(1)}},S_{j_{\sigma(1)}+j_{\sigma(2)}},\ldots,S_{j_{\sigma(1)}+\ldots+j_{\sigma(k)}})]\\
&=
\sum_{(j_1,\ldots,j_k) \in J_{n,k}}
\P[0\notin \conv (S_{j_{1}},S_{j_{1}+j_{2}},\ldots,S_{j_{1}+\ldots+j_{k}})]\\
&=\sum_{1\leq i_1<\ldots<i_k\leq n}
\P[0\notin \conv (S_{i_{1}},S_{i_{2}},\ldots,S_{i_{k}})].
\end{aligned}
\end{multline*}
To complete the proof, observe that the right-hand side is nothing but $2 \E M_{n,k}^{(d)}$; see~\eqref{eq:def_M_n_k}.
\end{proof}

\subsection{Proof of Theorem~\ref{theo:arcsine_multidim_Br}}
The proof is based on the following result obtained in~\cite{KVZ15}:
\begin{theorem}[\cite{KVZ15}]\label{theo:absorption_Br}
Let $(S_i)_{i=1}^n$, $n\geq 2$,  be a random bridge in $\R^d$ satisfying assumptions $(\text{Br})$, $(\text{Ex})$, $(\text{GP}')$. Then,
$$
\P[0\notin \conv (S_1,\ldots,S_{n-1})] = \frac 2 {n!}  \left(\stirling{n}{d} + \stirling{n}{d-2} +\ldots\right).
$$
\end{theorem}

\begin{proof}[Proof of Theorem~\ref{theo:arcsine_multidim_Br}]
Take some $1\leq i_1 < \ldots < i_k \leq n-1$ and put $i_0:=0$, $i_{k+1}:=n$.  We subdivide the collection $(\xi_1,\ldots,\xi_n)$  into $k+1$ groups of lengths
$$
j_1:=i_1>0, \;\; j_2:=i_2-i_1>0,\;\; \ldots, \;\; j_{k+1}:=n-i_{k}>0
$$
as follows:
\begin{equation}\label{eq:groups_2}
\underbrace{(\xi_1,\ldots,\xi_{i_1})}_{\text{length $j_1$}},
\;\;
\underbrace{(\xi_{i_1+1},\ldots,\xi_{i_2})}_{\text{length $j_2$}},
\;\; \ldots,\;\;
%\underbrace{(\xi_{i_{k+1} +1},\ldots, \xi_{i_k})}_{\text{length $j_k$}},
%\;\;
\underbrace{(\xi_{i_{k}+1},\ldots,\xi_n)}_{\text{length $j_{k+1}$}}.
\end{equation}

Denote by $\xi_l'$ the sum of the $\xi_i$'s in the $l$-th group above, that is
$$
\xi_l' := \xi_{i_{l-1} +1} +\ldots + \xi_{i_l}, \quad l=1,\ldots, k+1.
$$
Let $\tau$ be a random permutation that is uniformly distributed on $\Sym(k+1)$ and independent with $(\xi_1,\ldots,\xi_n)$. We reshuffle the $\xi_i'$'s as follows:
$$
\eta_i:=\xi_{\tau(i)}', \quad i=1,\ldots, k+1.
$$
By the construction, the $k+1$ random vectors $\eta_i$ satisfy the exchangeability assumtpion $(\text{Ex})$ (with $k+1$ substituted for $n$).

Take any permutation $\sigma$  of the set $\{1,\ldots,k+1\}$ and
permute the $k+1$ groups of the above list \eqref{eq:groups_2} according to $\sigma$.
The exchangeability assumption $(\text{Ex})$  implies that this does not change the distribution of $(\xi_1,\ldots,\xi_n)$, that is
$$
(\xi_1,\ldots,\xi_{n})
\eqdistr
(\underbrace{\xi_{i_{\sigma(1)-1} + 1},\ldots, \xi_{i_{\sigma(1)}}}_{\text{length $j_{\sigma(1)}$}},
\underbrace{\xi_{i_{\sigma(2)-1} + 1},\ldots, \xi_{i_{\sigma(2)}}}_{\text{length $j_{\sigma(2)}$}},
\ldots,
\underbrace{\xi_{i_{\sigma(k+1)-1} + 1},\ldots, \xi_{i_{\sigma(k+1)}}}_{\text{length $j_{\sigma(k+1)}$}}).
$$
The sum of random vectors in the $l$-th group is $\xi_{\sigma(l)}'$, which gives the distributional identity
\begin{equation}
\label{eq: distr = 2}
(S_{j_{\sigma(1)}},S_{j_{\sigma(1)}+j_{\sigma(2)}},\ldots,S_{j_{\sigma(1)}+\ldots+j_{\sigma(k+1)}}) \eqdistr (\xi_{\sigma(1)}', \xi_{\sigma(1)}' + \xi_{\sigma(2)}', \ldots, \xi_{\sigma(1)}' + \ldots + \xi_{\sigma(k+1)}' ),
\end{equation}
which is analogous to \eqref{eq: distr =}. Note that the last vectors in both sides equal zero a.s. since $j_1 + \ldots + j_{k+1}=n$ and $S_n=0$ a.s.

Now introduce the partial sums
$$T_i:= \eta_1 + \ldots +\eta_i,  \quad i=1,\ldots, k+1.$$
It follows from \eqref{eq: distr = 2} that $T_{k+1}=S_n=0$ a.s. and the random bridge $(T_i)_{i=1}^{k+1}$ satisfies the general position assumption  $(\text{GP}')$ (with $k+1$ substituted for $n$) since  the random bridge $(S_i)_{i=1}^n$ satisfies $(\text{GP}')$. Applying Theorem~\ref{theo:absorption_Br} to $(T_i)_{i=1}^{k+1}$, we obtain
\begin{multline}
\label{eq: symmetrized2}
\frac 1 {(k+1)!} \sum_{\sigma \in \Sym(k+1)}
\P[0\notin \conv (S_{j_{\sigma(1)}},S_{j_{\sigma(1)}+j_{\sigma(2)}},\ldots,S_{j_{\sigma(1)}+\ldots+j_{\sigma(k)}})]
\\
\begin{aligned}
&= \P[0\notin \conv (T_1, \ldots, T_k)] \\
&= \frac 2 {(k+1)!}  \left(\stirling{k+1}{d} + \stirling{k+1}{d-2} +\ldots\right).
\end{aligned}
\end{multline}

Now we take the sum over all $(j_1,\ldots,j_{k+1})$ from the set
$$
J_{n,k}^* := \{(j_1,\ldots,j_{k+1})\in \N^{k+1}\colon j_1+\ldots+j_{k+1} = n\}.
$$
Since the cardinality of $J_{n,k}^*$ is $\binom {n-1}k$ and the last expression in \eqref{eq: symmetrized2} does not depend on $(j_1,\ldots,j_{k+1})$, we obtain
\begin{multline*}
\binom {n-1}{k} \frac 2 {(k+1)!}  \left(\stirling{k+1}{d} + \stirling{k+1}{d-2} +\ldots\right)
\\
\begin{aligned}
&=
\frac 1 {(k+1)!} \sum_{\sigma \in \Sym(k+1)}  \sum_{(j_1,\ldots,j_{k+1}) \in J_{n,k}^*}
\P[0\notin \conv (S_{j_{\sigma(1)}},S_{j_{\sigma(1)}+j_{\sigma(2)}},\ldots,S_{j_{\sigma(1)}+\ldots+j_{\sigma(k)}})]\\
&=
\sum_{(j_1,\ldots,j_{k+1}) \in J_{n,k}^*}
\P[0\notin \conv (S_{j_{1}},S_{j_{1}+j_{2}},\ldots,S_{j_{1}+\ldots+j_{k}})]\\
&=\sum_{1\leq i_1<\ldots<i_k\leq n-1}
\P[0\notin \conv (S_{i_{1}},S_{i_{2}},\ldots,S_{i_{k}})].
\end{aligned}
\end{multline*}
To complete the proof, observe that the right-hand side is nothing but $2 \E M_{n,k}^{(d)}$; see~\eqref{eq:def_M_n_k_Br}.
\end{proof}

\section*{Acknowledgments}
We would like to thank the referee for his/her stimulating comments and suggestions.

\bibliography{arcsine_multidim_bib}
\bibliographystyle{plain}

\end{document}